\documentclass{amsart} 

\newtheorem{theorem}{Theorem}[section]
\newtheorem{lemma}[theorem]{Lemma}
\newtheorem{proposition}[theorem]{Proposition}
\newtheorem{corollary}[theorem]{Corollary}

\theoremstyle{definition}
\newtheorem{definition}[theorem]{Definition}
\newtheorem{example}[theorem]{Example}
\newtheorem{remark}[theorem]{Remark}

\newcommand{\Z}{\mathbb{Z}}

\usepackage{graphicx} 

\begin{document}

\title[Writhe polynomials and shell moves]
{Writhe polynomials and shell moves \\
for virtual knots and links}

\author{Takuji NAKAMURA}
\address{Department of Engineering Science, 
Osaka Electro-Communication University,
Hatsu-cho 18-8, Neyagawa, Osaka 572-8530, Japan}
\email{n-takuji@osakac.ac.jp}

\author{Yasutaka NAKANISHI}
\address{Department of Mathematics, Kobe University, 
Rokkodai-cho 1-1, Nada-ku, Kobe 657-8501, Japan}
\email{nakanisi@math.kobe-u.ac.jp}

\author{Shin SATOH}
\address{Department of Mathematics, Kobe University, 
Rokkodai-cho 1-1, Nada-ku, Kobe 657-8501, Japan}
\email{shin@math.kobe-u.ac.jp}

\renewcommand{\thefootnote}{\fnsymbol{footnote}}
\footnote[0]{
The first author is partially supported by 
JSPS Grants-in-Aid for Scientific Research (C), 
17K05265. 
The second author is partially supported by 
JSPS Grants-in-Aid for Scientific Research (C), 
19K03492. 
The third author is partially supported by 
JSPS Grants-in-Aid for Scientific Research (C), 
19K03466.}


\renewcommand{\thefootnote}{\fnsymbol{footnote}}





\maketitle


\begin{abstract} 
The writhe polynomial is a fundamental invariant 
of an oriented virtual knot. 
We introduce a kind of local moves for oriented virtual knots 
called shell moves. 
The first aim of this paper is to prove that 
two oriented virtual knots have 
the same writhe polynomial if and only if 
they are related by 
a finite sequence of shell moves. 
The second aim of this paper is 
to classify oriented $2$-component virtual links 
up to shell moves 
by using several invariants of virtual links. 
\end{abstract}

\section{Introduction}\label{sec1}

Several invariants of classical knots correspond to local moves. 
For example, 
two classical knots have the same Arf invariant 
if and only if they are related by a fintie sequence of 
pass moves \cite{Kau1, Kau2}. 
Such a correspondence reveals a relationship between 
algebraic and geometric structures of classical knots. 
A similar result is known in virtual knot theory. 
Two virtual knots have the same odd writhe  
if and only if they are related 
by a finite sequence of $\Xi$-moves \cite{ST}.

The writhe polynomial $W_K(t)$ of a virtual knot $K$ 
is a stronger invariant than the odd  writhe $J(K)$: 
For an integer $n\ne 0$, 
the $n$-writhe $J_n(K)$ of $K$ is defined 
by using the index of a chord of a Gauss diagram, 
and the writhe polynomial and odd writhe 
are described by 
$$W_K(t)=\sum_{n\ne 0} J_n(K)t^n-
\sum_{n\ne 0}J_n(K) \mbox{ and }
J(K)=\sum_{n:{\rm odd}}J_n(K).$$
In this paper 
we introduce two kinds of local moves called shell moves for virtual knots, 
which are defined by using 
Gauss diagrams as shown in Figure~\ref{fig101}. 
The precise definition is given in Section~\ref{sec2}. 
Then we prove the following.

\begin{figure}[htb]
\begin{center}
\includegraphics[bb=0 0 319 46]{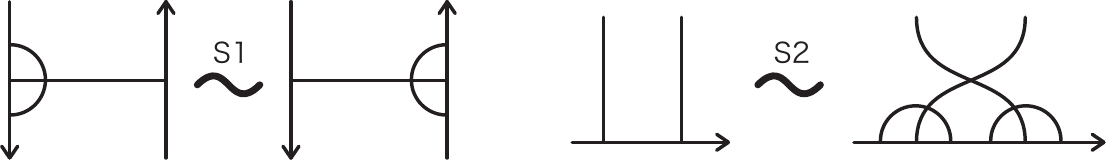}
\caption{Shell moves}
\label{fig101}
\end{center}
\end{figure}

\begin{theorem}\label{thm11}
For two oriented virtual knots $K$ and $K'$, 
the following are equivalent. 
\begin{itemize}
\item[{\rm (i)}] 
$W_K(t)=W_{K'}(t)$. 
\item[{\rm (ii)}] 
$K$ and $K'$ are related by a finite sequence of shell moves. 
\end{itemize} 
\end{theorem}

We extend the shell moves 
to oriented $2$-component virtual links. 
For an oriented $2$-component virtual link 
$L=K_1\cup K_2$, 
there are several invariants such as 
$\lambda(L)={\rm Lk}(K_1,K_2)-{\rm Lk}(K_2,K_1)$, 
$J_n(K_1;L)$, $J_n(K_2;L)$, 
and $F(L)$ whose precise definitions 
will be given later. 
By using these invariants, 
we classify oriented $2$-component virtual links 
up to shell moves. 
The situations are slightly different 
according to $\lambda(L)$.

\begin{theorem}\label{thm12}
Let $L=K_1\cup K_2$ and $L'=K_1'\cup K_2'$ 
be oriented $2$-component virtual links with 
$\lambda(L)=\lambda(L')=0$. 
Then $L$ and $L'$ are related by 
a finite sequence of shell moves 
if and only if 
\begin{itemize} 
\item[{\rm (i)}] 
$J_n(K_1;L)=J_n(K_1';L')$ for any $n\ne 0, 1$, 
\item[{\rm (ii)}] 
$J_n(K_2;L)=J_n(K_2';L')$ for any $n\ne 0, 1$, and 
\item[{\rm (iii)}] 
$F(L)=F(L')$. 
\end{itemize}
\end{theorem}

\begin{theorem}\label{thm13}
Let $L=K_1\cup K_2$ and $L'=K_1'\cup K_2'$ 
be oriented $2$-component virtual links with 
$\lambda(L)=\lambda(L')=1$. 
Then $L$ and $L'$ are related by 
a finite sequence of shell moves 
if and only if 
\begin{itemize} 
\item[{\rm (i)}] 
$J_n(K_1;L)=J_n(K_1';L')$ for any $n\ne 0, 1,-1$, 
\item[{\rm (ii)}] 
$J_n(K_2;L)=J_n(K_2';L')$ for any $n\ne 0, 1,2$, and 
\item[{\rm (iii)}] 
$F(L)=F(L')$. 
\end{itemize}
\end{theorem}

\begin{theorem}\label{thm14}
Let $L=K_1\cup K_2$ and $L'=K_1'\cup K_2'$ 
be oriented $2$-component virtual links with 
$\lambda=\lambda(L)=\lambda(L')\geq 2$. 
Then $L$ and $L'$ are related by 
a finite sequence of shell moves 
if and only if 
\begin{itemize} 
\item[{\rm (i)}] 
$J_n(K_1;L)=J_n(K_1';L')$ for any $n\ne 0, 1,-\lambda,-\lambda+1$, 
\item[{\rm (ii)}] 
$J_n(K_2;L)=J_n(K_2';L')$ for any $n\ne 0, 1,\lambda,\lambda+1$, 
\item[{\rm (iii)}] 
$F(L)=F(L')$, and 

\item[{\rm (iv)}] 
$J_1(K_1;L)+J_{-\lambda+1}(K_1;L)
+J_1(K_2;L)+J_{\lambda+1}(K_2;L)$

\hfill
$=J_1(K_1';L')+J_{-\lambda+1}(K_1';L')
+J_1(K_2';L')+J_{\lambda+1}(K_2';L')$.
\end{itemize}
\end{theorem}

This paper is organized as follows. 
In Section~\ref{sec2}, 
we introduce the notions of 
shells, shell moves, and snails in Gauss diagrams. 
We say that two Gauss diagrams are S-equivalent 
if they are related by a finite sequence of shell moves. 
We prove that any Gauss diagram 
is S-equivlanet to the one consisting of 
several snails. 
In Section~\ref{sec3}, 
we study a relationship between 
shell moves and writhe polynomials, 
and prove Theorem~\ref{thm11}. 
In Section~\ref{sec4}, 
we study shell moves for  
oriented $2$-component virtual links, 
and prove that any Gauss diagram 
is S-equivalent to the one in standard form. 
In Section~\ref{sec5}, 
we introduce several kinds of invariants 
of oriented $2$-component virtual links, 
and prove Theorems~\ref{thm12}--\ref{thm14}. 
In the last section, 
we give a relationship among the invariants 
in Section~\ref{sec5}, 
and prove that there is no relationship other than it.

\section{Shell moves}\label{sec2}

A {\it Gauss diagram} $G$ is a disjoint union of oriented circles 
equipped with a finite number of oriented and signed chords 
spanning the circles. 
For a chord with sign $\varepsilon$, 
we give signs $-\varepsilon$ and $\varepsilon$ 
to the initial and terminal endpoints of 
the chord, respectively. 
By definition, 
if two of three kinds of informations of a chord $\gamma$ 
--- the sign of $\gamma$, the orientation of $\gamma$, 
and the signs of endpoints of $\gamma$ --- 
are given, 
then the other is determined. 
We say that a chord $\gamma$ of $G$ is a {\it self-chord} 
if both endpoints of $\gamma$ belong to 
the same circle of $G$, 
and otherwise a {\it nonself-chord}. 
See Figure~\ref{fig201}. 
A self-chord $\gamma$ is {\it free} 
if the endpoints of $\gamma$ are adjacent on the circle. 
A Gauss diagram is called {\it empty} 
if it has no chord.

\begin{figure}[htb]
\begin{center}
\includegraphics[bb=0 0 296 60]{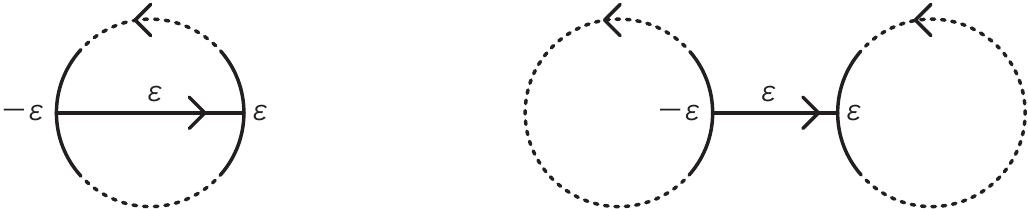}
\caption{Self- and nonself-chords of a Gauss diagram}
\label{fig201}
\end{center}
\end{figure}

Virtual knot theory is introduced by Kauffman \cite{Kau3}. 
A virtual link is 
an equivalence class of virtual link diagrams 
up to Reidemeister moves R1--R7. 
Furthermore,  
a virtual link diagram is an equivalence class of 
Gauss diagrams up to Reidemeister moves R4--R7. 
In this sense, 
a virtual link is an equivalence class of Gauss diagrams 
up to Reidemeister moves R1--R3 (cf.~\cite{GPV, Kau3}). 
In Figure~\ref{fig202}, 
we illustrate Reidemeister moves R1--R3 
with $\varepsilon=\pm$. 
Though there are many types of R3-moves 
with respect to orientations and signs of chords, 
it is enough to give just one type of R3-move as in the figure; 
for Polyak gives 
a minimal set of oriented Reidemeister moves \cite{Pol}.

\begin{figure}[htb]
\begin{center}
\includegraphics[bb=0 0 289 216]{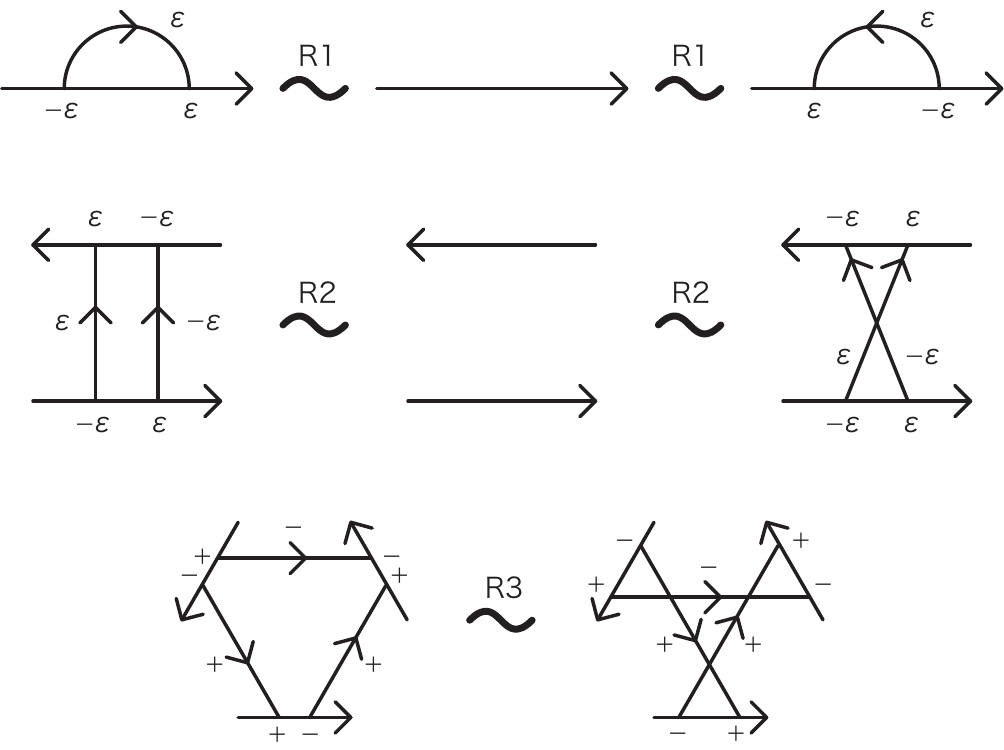}
\caption{Reidemeister moves for Gauss diagrams}
\label{fig202}
\end{center}
\end{figure}

A {\it $\mu$-component} virtual link 
is represented by a Gauss diagram 
with $\mu$ circles. 
In particular, 
a $1$-component virtual link is called a {\it virtual knot}. 
The {\it trivial} $\mu$-component virtual link 
is represented by the empty Gauss diagram 
consisting of $\mu$ circles.

\begin{definition}\label{def21} 
Let $\gamma$ be a self- or nonself-chord of $G$. 
The {\it shells} for $\gamma$ are self-chords 
which surround an endpoint of $\gamma$ in parallel 
such that 
if the endpoint of $\gamma$ has positive (or negative) sign, 
then the orientation of shells are 
the same as (or opposite to) that of the circle.  
See Figure~\ref{fig203}. 
\end{definition}

\begin{figure}[htb]
\begin{center}
\includegraphics[bb=0 0 163 47]{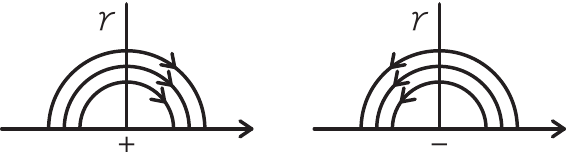}
\caption{Shells for $\gamma$}
\label{fig203}
\end{center}
\end{figure}

Since the orientation of shells is 
determined by the sign of the endpoint of $\gamma$, 
we sometimes omit the orientation of $\gamma$. 
The notion of a shell in this paper 
is slightly different from that of an anklet in \cite{NNS}.

We introduce two kinds of deformations 
on Gauss diagrams as follows. 

\begin{definition}\label{def22} 
Let $G$ be a Gauss diagram. 
\begin{itemize}
\item[(i)] 
A {\it shell move} S1 for $G$ 
is a deformation 
which slides a shell for a chord 
to another side of the chord 
with keeping the sign of the shell. 
See  the left of Figure~\ref{fig204}. 
\item[(ii)] 
A {\it shell move} S2 for $G$ is a deformation 
which changes the adjacent endpoints 
of a pair of chords with adding a shell to each chord 
as shown in the right of the figure. 
\item[(iii)] 
Two Gauss diagrams $G$ and $G'$ are {\it S-equivalent} 
if $G$ is related to $G'$ by 
a finite sequence of Reidemeister moves 
R1--R3 and shell moves S1 and S2. 
We denote it by $G\sim G'$. 
\item[(iv)] 
Two oriented virtual links are {\it S-equivalent} 
if their Gauss diagrams are S-equivalent. 
\end{itemize}
\end{definition}

\begin{figure}[htb]
\begin{center}
\includegraphics[bb=0 0 326 51]{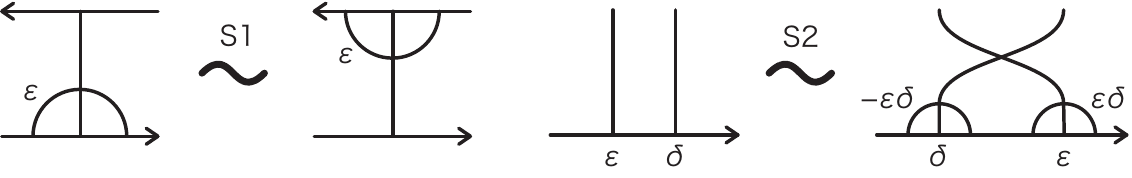}
\caption{Shell moves S1 and S2}
\label{fig204}
\end{center}
\end{figure}

Let $\gamma$ be a chord of a Gauss diagram, 
and $k$ the sum of signs of shells for $\gamma$. 
We can bunch all the shells at one endpoint of 
$\gamma$ by using S1-moves, 
and then cancel them pairwise by R2-moves 
so that we obtain $k$ shells 
with positive sign for $k\geq 0$ 
or $-k$ shells with negative sign for $k<0$. 
In this sense, 
the algebraic number of shells for a chord 
is uniquely determined up to S-equivalence. 
See Figure~\ref{fig205}.

\begin{figure}[htb]
\begin{center}
\includegraphics[bb=0 0 307 69]{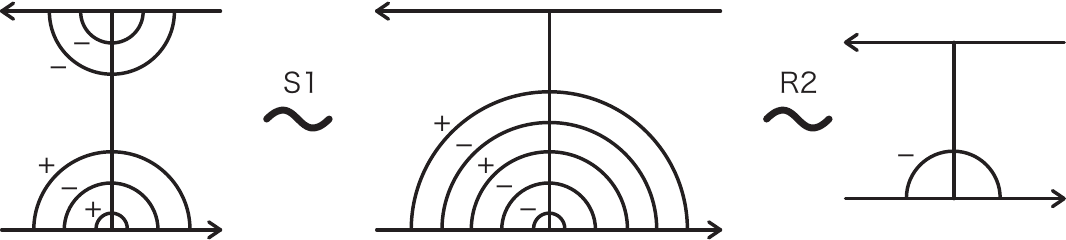}
\caption{The algebraic number of shells for a chord}
\label{fig205}
\end{center}
\end{figure}

\begin{lemma}\label{lem23}
If two Gauss diagrams are related by a deformation 
$(1)$ or $(2)$ 
as shown in {\rm Figure~\ref{fig206}}, 
then they are S-equivalent. 
In the figure, we indicate the algebraic number of shells 
for each chord. 
\end{lemma}

\begin{figure}[htb]
\begin{center}
\includegraphics[bb=0 0 335 122]{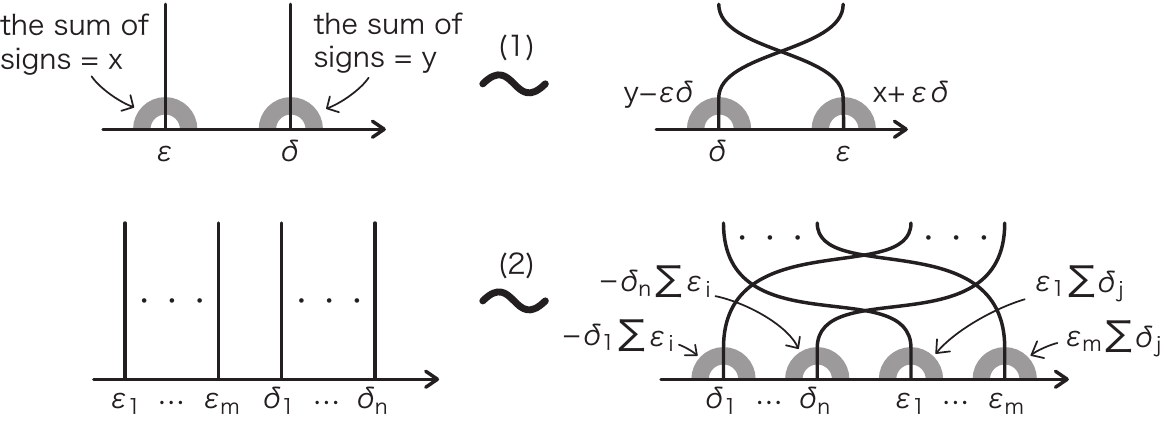}
\caption{S-equivalent Gauss diagrams in Lemma~\ref{lem23}}
\label{fig206}
\end{center}
\end{figure}

\begin{proof}
(1) 
This is a generalization of an S2-move. 
We first apply S1-moves to transfer shells 
to the other side of each chord 
and then perform an S2-move. 
Then we apply S1-moves 
to get the original position of shells. 

(2) 
This can be proved by (1) repeatedly. 
\end{proof}

\begin{corollary}\label{cor24} 
If two Gauss diagrams are related by a deformation 
$(1)$--$(4)$ 
as shown in {\rm Figure~\ref{fig207}}, 
then they are S-equivalent. 
Here, $P$ and $Q$ are portions of whole chords. 
\end{corollary}

\begin{figure}[htb]
\begin{center}
\includegraphics[bb=0 0 226 166]{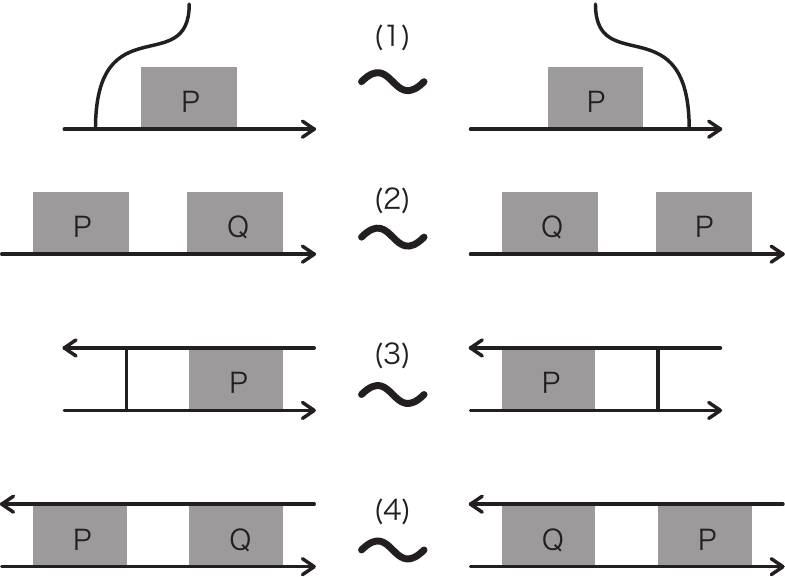}
\caption{S-equivalent Gauss diagrams in Corollary~\ref{cor24}}
\label{fig207}
\end{center}
\end{figure}

\begin{proof}
(1) 
Since the sum of signs of 
endpoints of chords in $P$ is equal to $0$, 
the Gauss diagrams are S-equivalent by Lemma~\ref{lem23}(2). 
We remark that although each chord in $P$ 
gets a pair of shells on both endpoints, 
they have opposite signs 
and can be canceled. 

(2) The deformation is 
realized by the combination of (1)'s. 

(3) We apply Lemma~\ref{lem23}(2) twice 
as shown in Figure~\ref{fig208}. 
Then we see that any new  shells can be canceled 
by S1- and R2-moves.

(4) This can be proved similarly to (3). 
\end{proof}

\begin{figure}[htb]
\begin{center}
\includegraphics[bb=0 0 343 78]{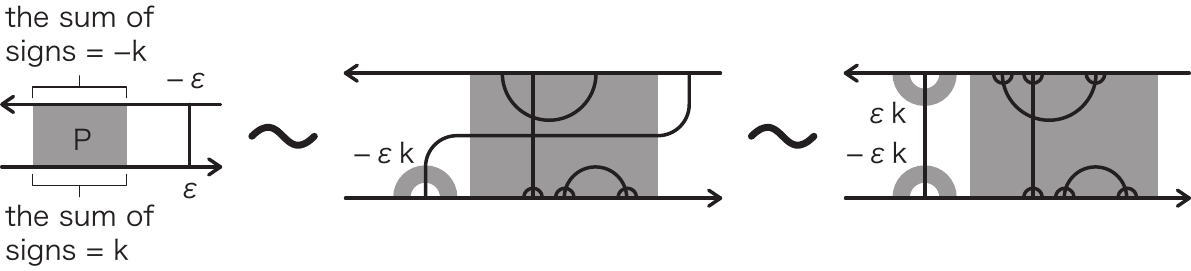}
\caption{Proof of Corollary~\ref{cor24}(3)}
\label{fig208}
\end{center}
\end{figure}

\begin{definition}\label{def26}
Let $G$ be a Gauss diagram 
with $\mu$ circles $C_1,\dots,C_{\mu}$. 
\begin{itemize}
\item[(i)] 
For $n\in{\Z}$ and $1\leq i\leq \mu$, 
a positive or negative {\it $n$-snail of type $i$} is 
the portion of a self-chord $\gamma$ with $|n|$ shells 
such that $\gamma$ is spanning the circle $C_i$ 
and the sign of $\gamma$ is positive or negative, 
respectively, 
as shown in the left of Figure~\ref{fig209}.
We denote it by $+S_i(n)$ or $-S_i(n)$. 
We remark that 
$\pm S_i(0)$ consists of a free chord. 

\item[(ii)] 
For $n\in{\Z}$ and $1\leq i\ne j\leq \mu$, 
a positive or negative 
{\it $n$-snail of type $(i,j)$} 
is the portion of a nonself-chord $\gamma$ with $|n|$ shells 
such that $\gamma$ is spanning between 
the circles $C_i$ and $C_j$ 
oriented from $C_i$ to $C_j$, 
and the sign of $\gamma$ is positive or negative, 
respectively, 
as shown in the right of the figure. 
We denote it by $+S_{ij}(n)$ or $-S_{ij}(n)$. 
\end{itemize}
\end{definition}

\begin{figure}[htb]
\begin{center}
\includegraphics[bb=0 0 304 78]{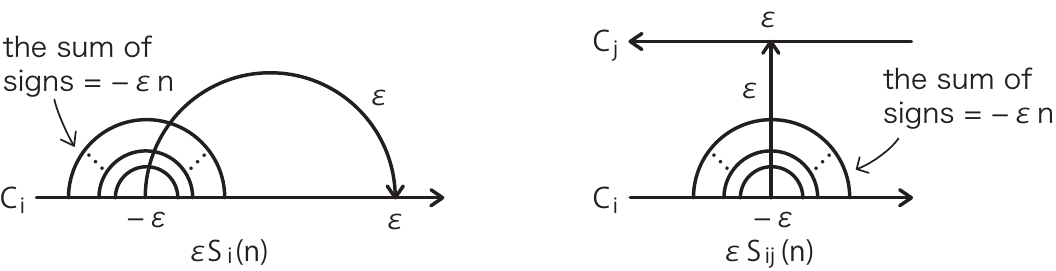}
\caption{The $n$-snails $\varepsilon S_i(n)$ and 
$\varepsilon S_{ij}(n)$}
\label{fig209}
\end{center}
\end{figure}

\begin{lemma}\label{lem26} 
If two Gauss diagrams are related by a deformation 
$(1)$--$(4)$ 
as shown in {\rm Figure~\ref{fig210}}, 
then they are S-equivalent. 
\end{lemma}

\begin{figure}[htb]
\begin{center}
\includegraphics[bb=0 0 281 110]{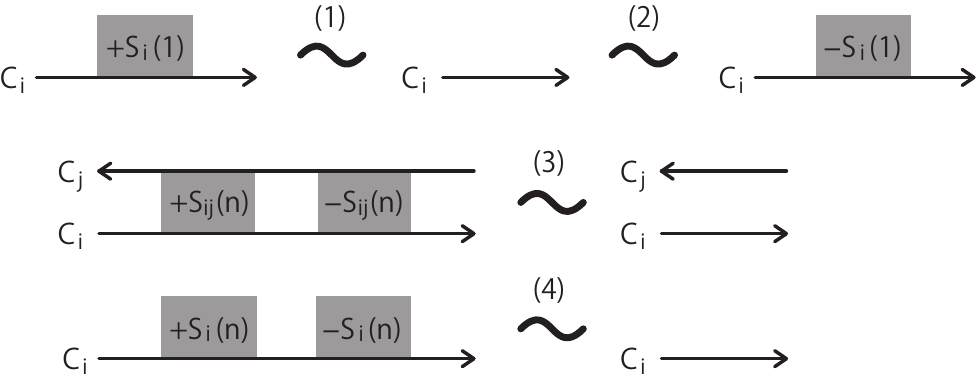}
\caption{S-equivalent Gauss diagrams in Lemma~\ref{lem26}}
\label{fig210}
\end{center}
\end{figure}

\begin{proof}
(1) and (2) 
The positive $1$-snail $+S_i(1)$ is 
related to $-S_i(1)$ by an S1-move, 
which is eliminated by an R2-move 
as shown in Figure~\ref{fig211}.

\begin{figure}[htb]
\begin{center}
\includegraphics[bb=0 0 271 38]{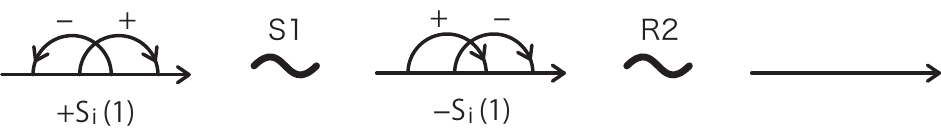}
\caption{Proofs of Lemma~\ref{lem26}(1) and (2)}
\label{fig211}
\end{center}
\end{figure}

(3) 
The concatenation of $+S_{ij}(n)$ and $-S_{ij}(n)$ 
is S-equivalent to that of 
$n$ copies of $+S_i(1)$ for $n>0$ 
or $n$ copies of canceling pairs for $n<0$ 
by using Lemma~\ref{lem23}(2) and 
Corollary~\ref{cor24}(1). 
See Figure~\ref{fig212}, 
where $\varepsilon$ is the sign of $n$; 
that is, $n=\varepsilon|n|$. 
Eventually it is S-equivalent to the empty 
by (1) or R2-moves.

\begin{figure}[htb]
\begin{center}
\includegraphics[bb=0 0 314 159]{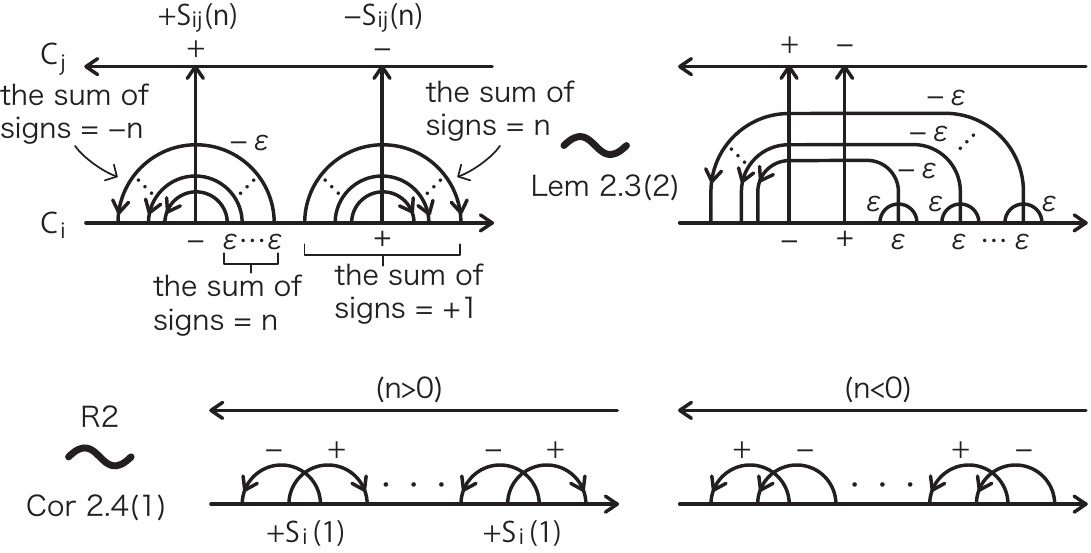}
\caption{Proof of Lemma~\ref{lem26}(3)}
\label{fig212}
\end{center}
\end{figure}

(4) 
We deform the concatenation of 
$+S_i(n)$ and $-S_i(n)$ 
as shown in Figure~\ref{fig213}. 
Then we can apply the same deformation 
used in the proof of (3) 
so that it is S-equivalent to the empty. 
\end{proof}

\begin{figure}[htb]
\begin{center}
\includegraphics[bb=0 0 299 47]{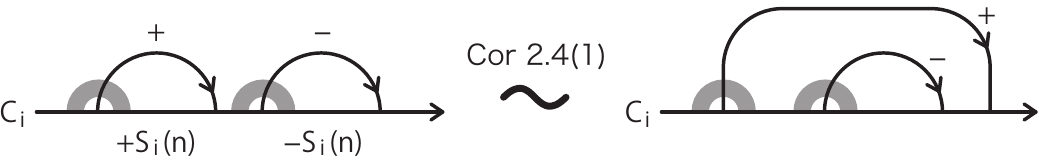}
\caption{Proof of Lemma~\ref{lem26}(4)}
\label{fig213}
\end{center}
\end{figure}

\begin{proposition}\label{prop27} 
Any Gauss diagram of an oriented $\mu$-component virtual link 
is S-equivalent to a Gauss diagram $G$ with $\mu$ circles 
$C_1,\dots,C_{\mu}$ 
which satisfies the following conditions. 
{\rm Figure~\ref{fig214}} shows the case $\mu=3$. 
\begin{itemize}
\item[{\rm (i)}] 
The chords of $G$ form a finite number of snails. 

\item[{\rm (ii)}]
There is an arc $\alpha_i$ on each $C_i$ 
such that 
all snails of type $i$ spans $\alpha_i$. 

\item[{\rm (iii)}] 
All snails of type $(i,j)$ 
spans 
$(C_i\setminus\alpha_i)\cup(C_j\setminus\alpha_j)$
in parallel. 

\item[{\rm (iv)}] 
There is no snails $\pm S_i(0)$ or $\pm S_i(1)$ for any $i$. 

\item[{\rm (v)}] 
There is no pair of snails 
$+S_i(n)$ and $-S_i(n)$ for any $i$ and $n$. 
\item[{\rm (vi)}] 
There is no pair of snails  
$+S_{ij}(n)$ and $-S_{ij}(n)$ for any $i\ne j$ and $n$. 
\end{itemize}
\end{proposition}

\begin{figure}[htb]
\begin{center}
\includegraphics[bb=0 0 163 132]{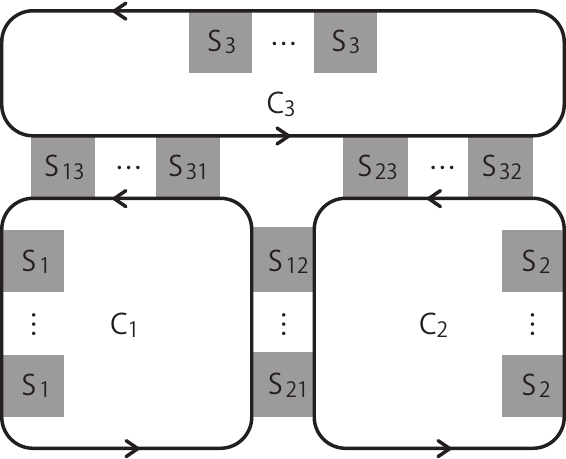}
\caption{A Gauss diagram with three circles}
\label{fig214}
\end{center}
\end{figure}

\begin{proof}
For each self-chord $\gamma$ spanning $C_i$, 
we slide the initial endpoint of $\gamma$ along $C_i$ 
with respect to the orientation of $C_i$ 
by using Lemma~\ref{lem23}(1) 
so that the initial endpoint of $\gamma$ 
is adjacent to the terminal 
with some shells. 
Then we obtain a snail $\pm S_i(n)$ on $C_i$ 
for some $n$. 
By Corollary~\ref{cor24}(1) and (2), 
we may assume that 
all snails span $\alpha_i$, 
and the endpoints of nonself-chords with shells on $C_i$ 
are contained in $C_i\setminus\alpha_i$. 
For the nonself-chords with shells 
between $C_i$ and $C_j$, 
we move the endpoints on $C_i\setminus\alpha_i$ 
by Lemma~\ref{lem23}(1) 
to obtain parallel snails of type $(i,j)$. 
Therefore we have the conditions (i)--(iii). 

The conditions (iv)--(vi) 
are derived from Lemma~\ref{lem26} directly. 
\end{proof}

We remark that by Corollary~\ref{cor24}(2) and (4), 
the Gauss diagrams which are different 
in the position of snails are all S-equivalent.

\section{The case $\mu=1$}\label{sec3}

Throughout this section, 
we consider an oriented virtual knot $K$ 
and its Gauss diagram $G$ 
consisting of a single circle $C$.

Let $P_i$ $(1\leq i\leq k)$ be portions of 
whole chords on $C$. 
We denote by 
$\left(\sum_{i=1}^k P_i\right)$ 
the Gauss diagram consisting of 
the concatenation of $P_i$'s. 
See Figure ~\ref{fig301}.

\begin{figure}[htb]
\begin{center}
\includegraphics[bb=0 0 163 42]{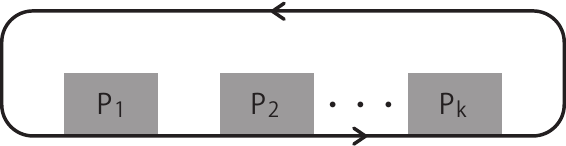}
\caption{The Gauss diagram $\left(\sum_{i=1}^k P_i\right)$}
\label{fig301}
\end{center}
\end{figure}

For integers $a$ and $n$, 
we denote by $aS(n)$ 
the concatenation of $a$ copies of $+S(n)$ for $a>0$, 
$-a$ copies of $-S(n)$ for $a<0$, 
and the empty for $a=0$. 
Here, we abbreviate $\pm S_1(n)$ to $\pm S(n)$ 
for simplicity. 
Then by Proposition~\ref{prop27} 
we have the following.

\begin{lemma}\label{lem31}
Any Gauss daigram of $K$ 
is $S$-equivalent to 
$\left(\sum_{n\ne 0,1}a_n S(n)\right)$ 
for some $a_n\in{\Z}$. 
\hfill $\Box$
\end{lemma}

Let $G$ be a Gauss diagram of a virtual knot $K$, 
and $\gamma$ a chord of $G$. 
The endpoints of  $\gamma$ divide the circle $C$ into two arcs. 
Let $\alpha$ be the arc oriented 
from the initial endpoint of $\gamma$ to the terminal. 
The {\it index} of $\gamma$ is the sum of signs 
of endpoints of chords on $\alpha$, 
and denoted by ${\rm Ind}(\gamma)\in{\Z}$ 
(cf.~\cite{Che, Kau, ST}).

For each integer $n$, 
we denote by $J_n(G)$ the sum of signs of all chords $\gamma$ 
with ${\rm Ind}(\gamma)=n$. 
If $n\ne 0$, then 
$J_n(G)$ does not depend on a particular choice of 
$G$ of $K$ \cite{ST}. 
It is called the {\it $n$-writhe} of $K$ 
and denoted by $J_n(K)$. 
The {\it writhe polynomial} of $K$ 
is defined by 
$$W_K(t)=\sum_{n\ne 0}J_n(K)t^n-
\sum_{n\ne 0}J_n(K)\in{\Z}[t,t^{-1}].$$
This invariant is introduced in several papers 
\cite{CG, Kau, ST} independently. 
A characterization of $W_K(t)$ is 
given as follows.

\begin{theorem}[\cite{ST}]\label{thm32}
For a Laurent polynomial $f(t)\in{\Z}[t,t^{-1}]$, 
the following are equivalent. 
\begin{itemize}
\item[(i)] 
There is a virtual knot $K$ such that 
$W_K(t)=f(t)$. 
\item[(ii)] 
$f(1)=f'(1)=0$. 
\end{itemize}
In particular, 
it holds that 
$J_1(K)=-\sum_{n\ne 0,1}nJ_n(K)$. 
\hfill$\Box$
\end{theorem}

\begin{example}\label{ex33}
Figure~\ref{fig302} shows an example of a Gauss diagram $G$ 
with five chords whose indices are surrounded by boxes. 
Let $K$ be the oriented virtual knot 
presented by $G$. 
It holds that 
$$J_n(K)=
\left\{
\begin{array}{rl}
1 & (n=3, -1), \\
-2 & (n=1), \mbox{ and}\\
0 & (n\ne 3,1,0,-1). 
\end{array}\right.$$
Therefore we have 
we have $W_K(t)=t^{-1}-2t+t^3$. 
\end{example}

\begin{figure}[htb]
\begin{center}
\includegraphics[bb=0 0 89 88]{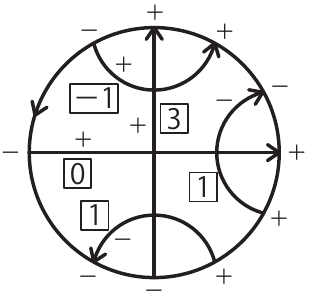}
\caption{A Gauss diagram with five chords}
\label{fig302}
\end{center}
\end{figure}

\begin{lemma}\label{lem34} 
Let $\gamma$ be a chord of $G$. 

\begin{itemize}
\item[{\rm (i)}] 
If $\gamma$ is a shell, 
then we have ${\rm Ind}(\gamma)=1$.

\item[{\rm (ii)}] 
If $\gamma$ is not a shell, 
then the index does not change under S-moves. 
\end{itemize}
\end{lemma}

\begin{proof}
(i) 
Since the sum of signs 
of endpoints of all chords 
is equal to $0$, 
it follows from the definition of a shell. 

(ii) This follows from the definition of the index of a chord. 
We remark that the indices of chords 
$\gamma$ and $\gamma'$ as shown in Figure~\ref{fig303} 
do not change under an S2-move. 
\end{proof}

\begin{figure}[htb]
\begin{center}
\includegraphics[bb=0 0 173 86]{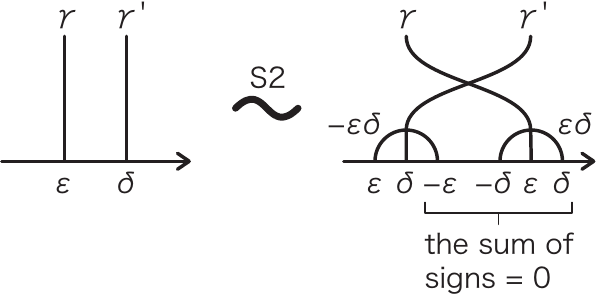}
\caption{A shell move S2}
\label{fig303}
\end{center}
\end{figure}

\begin{lemma}\label{lem35} 
Let $K$ be an oriented virtual knot. 
\begin{itemize}
\item[{\rm (i)}] 
The $n$-writhe $J_n(K)$ $(n\ne 0)$ 
is invariant under S-moves, 
and hence so is the writhe polynomial $W_K(t)$. 
\item[{\rm (ii)}] 
If $K$ is presented by a Gauss diagram 
given in {\rm Lemma~\ref{lem31}}, 
then we have 
$$W_K(t)=
\sum_{n\ne 0,1} a_nt^n-
\left(\sum_{n\ne 0,1}na_n\right)t
+\sum_{n\ne 0,1}(n-1)a_n. $$
\end{itemize}
\end{lemma}

\begin{proof}
(i) If a chord $\gamma$ satisfies 
${\rm Ind}(\gamma)\ne 1$, 
then the index of $\gamma$ does not change 
under S-moves, and hence 
$J_n(K)$ is invariant for $n\ne 0,1$ by Lemma~\ref{lem34}. 
Furthermore, 
since $J_1(K)=-\sum_{n\ne 0,1}nJ_n(K)$ 
by Theorem~\ref{thm32}, 
$J_1(K)$ is also invariant under S-moves. 

(ii) Since 
$J_n(K)=a_n \ (n\ne 0,1)$ and 
$J_1(K)=-\sum_{n\ne 0,1}na_n$ hold, 
we have the conclusion. 
\end{proof}

\begin{theorem}\label{thm36} 
Let $K$ and $K'$ be oriented virtual knots. 
If $W_K(t)=W_{K'}(t)$ holds, 
then $K$ and $K'$ are S-equivalent. 
\end{theorem}

\begin{proof}
By Lemma~\ref{lem31}, 
any Gauss diagrams of $K$ and $K'$ are S-equivalent to 
Gauss diagrams 
$$G=\left(\sum_{n\ne 0,1}a_n S(n)\right) \mbox{ and } 
G'=\left(\sum_{n\ne 0,1}a_n' S(n)\right),$$ 
respectively. 

By Lemma~\ref{lem35}(ii) and the assumption, 
we obtain 
$a_n=a_n'$ for any $n\ne 0,1$, 
and hence $G=G'$. 

\end{proof}

\begin{proof}[Proof of {\rm Theorem~\ref{thm11}}]
This follows from Lemma~\ref{lem35}(i) and 
Theorem~\ref{thm36}. 
\end{proof}

\begin{remark}\label{rem37}
We can give an alternative proof of 
Theorem~\ref{thm32} as follows: 
A Laurent polynomial $f(t)=\sum_{n\in{\Z}} a_nt^n$ 
satisfies $f(1)=f'(1)=0$ 
if and only if 
$a_1=-\sum_{n\ne 0,1}na_n$ 
and $a_0=\sum_{n\ne 0,1}(n-1)a_n$, 
that is, 
$$f(t)=
\sum_{n\ne 0,1} a_nt^n-
\left(\sum_{n\ne 0,1}na_n\right)t
+\sum_{n\ne 0,1}(n-1)a_n. $$
This polynomial is equal to 
$W_K(t)$ given in Lemma~\ref{lem35}(ii). 
\end{remark}

\begin{remark}\label{rem37}
The Jones polynomial of a virtual knot 
is invariant under an S1-move but not under an S2-move. 
The proof is easy and will be left to the reader. 
\end{remark}

\section{The case $\mu=2$ (geometric part)}\label{sec4}

Throughout Sections~\ref{sec4}--\ref{sec6}, 
we consider an oriented $2$-component virtual link 
$L=K_1\cup K_2$ 
and its Gauss diagram $G$ consisting of a pair of 
circles $C_1$ and $C_2$. 
By Proposition~\ref{prop27} 
we have the following.

\begin{lemma}\label{lem41}
Any Gauss daigram of $L$ 
is $S$-equivalent to a Gauss diagram 
$$\left(\sum_{n\ne 0,1}a_n S_1(n), 
\sum_{n\ne 0,1}b_n S_2(n); 
\sum_{m\in{\Z}}c_m S_{12}(m), 
\sum_{m\in{\Z}}d_m S_{21}(m) \right)$$ 
for some integers 
$a_n, b_n$ $(n\ne 0,1)$ and $c_m,d_m$ $(m\in{\Z})$ 
as shown in {\rm Figure~\ref{fig401}}. 
Here, the entries present the concatenations 
of snails of type $1$, $2$, $(1,2)$, and $(2,1)$, 
respectively. 
\hfill $\Box$
\end{lemma}

\begin{figure}[htb]
\begin{center}
\includegraphics[bb=0 0 129 78]{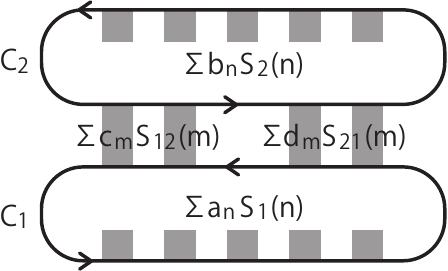}
\caption{A Gauss diagram of an oriented $2$-component virtual link}
\label{fig401}
\end{center}
\end{figure}

A nonself-chord of $G$ is called {\it of type $(i,j)$} 
for $1\leq i\ne j\leq 2$ 
if it is oriented from $C_i$ to $C_j$; 
that is, the initial and terminal endpoints belong to 
$C_i$ and $C_j$, respectively. 
The {\it $(i,j)$-linking number} of $L$, 
denoted by ${\rm Lk}(K_i,K_j)$, 
is the sum of signs of all nonself-chords 
of type $(i,j)$, 
which does not depend on a particular choice of 
$G$ for $L$. 
The {\it virtual linking number} of $L$ is defined 
by $\lambda(L)={\rm Lk}(K_1,K_2)-{\rm Lk}(K_2,K_1)$ 
\cite{Oka} (cf.~\cite{FK}).

\begin{lemma}\label{lem42}
Let $L=K_1\cup K_2$ be an oriented $2$-component virtual link. 
\begin{itemize}
\item[{\rm (i)}] 
The linking numbers ${\rm Lk}(K_1,K_2)$ and ${\rm Lk}(K_2,K_1)$ 
are invariant under S-moves, 
and hence so is the virtual linking number $\lambda(L)$. 
\item[{\rm (ii)}] 
If $L$ is presented by a Gauss diagram 
given in {\rm Lemma~\ref{lem41}}, 
then we have 
${\rm Lk}(K_1,K_2)=\sum_{m\in{\Z}}c_m$ and 
${\rm Lk}(K_2,K_1)=\sum_{m\in{\Z}}d_m$. 
\end{itemize}
\end{lemma}

\begin{proof}
(i) The invariance under S-moves 
is obtained by definition directly. 

(ii) 
Each snail $\pm S_{ij}(m)$ 
contains the unique nonself-chord with sign $\pm 1$. 
\end{proof}

For simplicity, 
we use the notation $\lambda=\lambda(L)$.

\begin{lemma}\label{lem43}
For any Gauss diagram of $L$, we have the following. 
\begin{itemize}
\item[{\rm (i)}] 
The sum of signs of endpoints of chords on $C_1$ 
is equal to $-\lambda$. 
\item[{\rm (ii)}] 
The sum of signs of endpoints of chords on $C_2$ 
is equal to $\lambda$. 
\end{itemize}
\end{lemma}

\begin{proof}
(i) 
Each self-chord of type $1$ 
does not contribute to the sum of signs on $C_1$. 
On the other hand, 
each nonself-chord of type $(1,2)$ (or type $(2,1))$ 
with sign $\varepsilon$ 
contributes to the sum on $C_1$ by $-\varepsilon$ 
(or $\varepsilon$). 
Therefore the sum on $C_1$ is equal to 
$-{\rm Lk}(K_1,K_2)+{\rm Lk}(K_2,K_1)=-\lambda$. 

(ii) By changing the roles between $K_1$ and $K_2$ in (i), 
the sum on $C_2$ is equal to 
$-{\rm Lk}(K_2,K_1)+{\rm Lk}(K_1,K_2)=\lambda$. 
\end{proof}

\begin{lemma}\label{lem44}
We have the following S-equivalent Gauss diagrams. 

\begin{itemize}
\setlength{\parskip}{2mm} 
  \setlength{\itemsep}{0cm} 
\item[{\rm (i)}]
$\left(P+S_1(-\lambda), Q;
\sum c_m S_{12}(m), 
\sum d_m S_{21}(m) \right)$

\hfill
$\sim\left(P, Q;
\sum c_m S_{12}(m-1), 
\sum d_m S_{21}(m+1) \right)$. 

\item[{\rm (ii)}] 
$\left(P+S_1(-\lambda+1), Q;
\sum c_m S_{12}(m), 
\sum d_m S_{21}(m) \right)$

\hfill
$\sim\left(P, Q;
\sum c_m S_{12}(m-1), 
\sum d_m S_{21}(m+1) \right)$. 

\item[{\rm (iii)}] 
$\left(P, Q+S_2(\lambda);
\sum c_m S_{12}(m), 
\sum d_m S_{21}(m) \right)$

\hfill
$\sim\left(P, Q;
\sum c_m S_{12}(m+1), 
\sum d_m S_{21}(m-1) \right)$. 

\item[{\rm (iv)}] 
$\left(P, Q+S_2(\lambda+1);
\sum c_m S_{12}(m), 
\sum d_m S_{21}(m) \right)$

\hfill
$\sim\left(P, Q;
\sum c_m S_{12}(m+1), 
\sum d_m S_{21}(m-1) \right)$. 
\end{itemize}
\end{lemma}

\begin{proof}
(i) 
In the Gauss diagram in the left hand side, 
let $\gamma$ be the self-chord of 
$+S_1(-\lambda)$ other than the shells. 
We move the terminal endpoint of $\gamma$ 
around $C_1$ with respect to the orientation of $C_1$. 

By Lemmas~\ref{lem23} and \ref{lem43}, 
the terminal endpoint of $\gamma$ 
gets shells such that the sum of signs is equal to $-\lambda$. 
Since the algebraic number of shells for $\gamma$ is equal to $0$, 
$\gamma$ becomes a free chord 
which can be removed by an R1-move. 
See Figure~\ref{fig402}. 

On the other hand, 
each snail $\pm S_{12}(m)$ changes into 
$\pm S_{12}(m-1)$ after the terminal endpoint of $c$ passes, 
and $\pm S_{21}(m)$ changes into $\pm S_{21}(m+1)$.

\begin{figure}[htb]
\begin{center}
\includegraphics[bb=0 0 291 88]{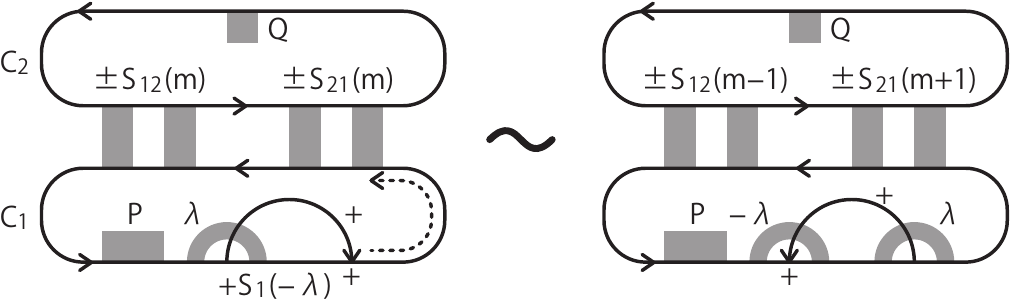}
\caption{Proof of Lemma~\ref{lem44}(i)}
\label{fig402}
\end{center}
\end{figure}

(ii) 
The proof is almost the same as that of (i). 
The difference is that the algebraic number 
of shells for $\gamma$ 
after moving the terminal endpoint of $\gamma$ 
around $C_1$ is equal to $-1$, 
the snail $+S_1(-\lambda+1)$ 
changes into a pair of chords 
which can be canceled by an R2-move. 

(iii) and (iv) These are obtained from (i) and (ii) 
by changing the roles of 
first and second components. 
\end{proof}

\begin{lemma}\label{lem45}
For any integer $M$, 
a Gauss diagram 
$$\left(P, Q;
\sum_{m\in{\Z}} c_m S_{12}(m), 
\sum_{m\in{\Z}} d_m S_{21}(m) \right)$$
is S-equivalent to the following. 

\begin{itemize}
\setlength{\parskip}{2mm} 
  \setlength{\itemsep}{0cm} 
\item[{\rm (i)}]
$\displaystyle{
\left(P, Q;
c_M S_{12}(M-c_M+\lambda)
+\sum_{m\ne M} c_m S_{12}(m-c_M), 
\sum_{m\in{\Z}} d_m S_{21}(m+c_M) \right)}$. 

\item[{\rm (ii)}] 
$\displaystyle{
\left(P, Q;
c_M S_{12}(M+c_M-\lambda)
+\sum_{m\ne M} c_m S_{12}(m+c_M), 
\sum_{m\in{\Z}} d_m S_{21}(m-c_M) \right)}$.

\item[{\rm (iii)}]
$\displaystyle{
\left(P, Q;
\sum_{m\in{\Z}} c_m S_{12}(m+d_M), 
d_M S_{21}(M-d_M-\lambda)
+\sum_{m\ne M} d_m S_{21}(m-d_M) \right)}$. 

\item[{\rm (iv)}]
$\displaystyle{
\left(P, Q;
\sum_{m\in{\Z}} c_m S_{12}(m-d_M), 
d_M S_{21}(M+d_M+\lambda)
+\sum_{m\ne M} d_m S_{21}(m+d_M) \right)}$.

\end{itemize}
\end{lemma}

\begin{proof}
(i) 
We move the endpoints of $c_MS_{12}(M)$ 
around $C_1$ with respect to the orientation of $C_1$. 
See Figure~\ref{fig403}. 
By Lemmas~\ref{lem23} and \ref{lem43}, 
$c_MS_{12}(M)$ 
changes into $c_M S_{12}(M-c_M+\lambda)$, 
and $c_m S_{12}(m)$ $(m\ne M)$ 
and $d_m S_{21}(m)$ change into 
$c_m S_{12}(m-c_M)$ and $d_m S_{21}(m+c_M)$, 
respectively. 

\begin{figure}[htb]
\begin{center}
\includegraphics[bb=0 0 291 78]{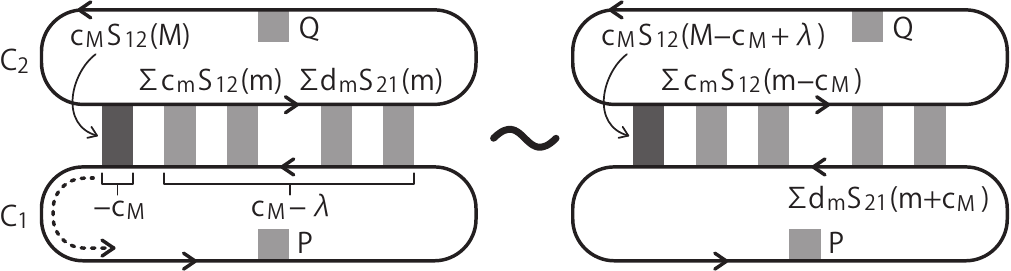}
\caption{Proof of Lemma~\ref{lem45}(i)}
\label{fig403}
\end{center}
\end{figure}

(ii) 
It is sufficient to move the endpoints of 
$c_M S_{12}(M)$ around $C_1$ 
with respect to the reverse orientation of $C_1$. 

(iii) and (iv) These are obtained from (i) and (ii) 
by changing the roles of first and second components. 
\end{proof}

If $\lambda(L)<0$, 
then by switching the roles of $K_1$ and $K_2$, 
the case reduces to $\lambda(L)>0$. 
In what follows, we may assume that 
$\lambda(L)\geq 0$.

\begin{proposition}\label{prop46}
Let $G$ be a Gauss diagram of $L$. 

\begin{itemize} 
\item[{\rm (i)}] 
If $\lambda\geq 1$, then 

$\displaystyle{
G\sim\left(\sum_{n\ne 0,1,-\lambda,-\lambda+1}a_n S_1(n), 
\sum_{n\ne 0,1,\lambda,\lambda+1}b_n S_2(n);\right.}$

\hfill
$\displaystyle{
\left.\sum_{m=0}^{\lambda-1}c_m S_{12}(p+m), 
\sum_{m=0}^{\lambda-1}d_m S_{21}(-p-m) \right)}$

\noindent
for some integers $a_n$ $(n\ne 0,1,-\lambda,-\lambda+1)$, 
$b_n$ $(n\ne 0,1,\lambda,\lambda+1)$, 
$c_m,d_m$ $(0\leq m\leq \lambda-1)$, and $p$. 

\item[{\rm (ii)}] 
In particular, if $\lambda=1$, then 

$$G\sim\left(\sum_{n\ne 0,1,-1}a_n S_1(n), 
\sum_{n\ne 0,1,2}b_n S_2(n);
c_0 S_{12}(0), d_0 S_{21}(0) \right)$$

\noindent
for some integers $a_n$ $(n\ne 0,1,-1)$, 
$b_n$ $(n\ne 0,1,2)$, 
$c_0$, and $d_0$.
\end{itemize}
\end{proposition}

\begin{proof}
(i) 
We may start a Gauss diagram 
of the form in Lemma~\ref{lem41}. 
By Lemma~\ref{lem44}, 
we can remove the snails  
$\pm S_1(-\lambda)$ and $\pm S_1(-\lambda+1)$ 
from the first entry, 
and $\pm S_2(\lambda)$ and $\pm S_2(\lambda+1)$ 
from the second entry. 
Moreover, by Lemma~\ref{lem45}, 
we see that there are integers $p$ and $q$ such that 
it is S-equivalent to a Gauss diagram 
$$G'=\left(P,Q; 
\sum_{m=0}^{\lambda-1}c_m S_{12}(p+m), 
\sum_{m=0}^{\lambda-1}d_m S_{21}(q-m) \right)$$
with 
$$P=\sum_{n\ne 0,1,-\lambda,-\lambda+1}a_n S_1(n)
\mbox{ and }
Q=\sum_{n\ne 0,1,\lambda,\lambda+1}b_n S_2(n)$$
for some integers $a_n$, $b_n$, $c_m$, $d_m$, $p$, and $q$.

Now we put $h(G')=p+q$. 
If $h(G')=0$, then the proof is completed. 
Assume that $h(G')>0$. 
The case $h(G')<0$ can be similarly proved. 
By applying Lemma~\ref{lem45}(ii) 
for $G'$ with $M=p+\lambda-1$, 
$G$ is S-equivalent to the Gauss diagram $G''$ 
given by 

$\displaystyle{
\left(P,Q; 
c_{p+\lambda-1}S_{12}(p+c_{p+\lambda-1}-1)+
\sum_{m=0}^{\lambda-2}c_m S_{12}
(p+c_{p+\lambda-1}+m), \right.}$

\hfill
$\displaystyle{
\left.
\sum_{m=0}^{\lambda-1}d_m S_{21}(q-c_{p+\lambda-1}-m) \right)}.$

\noindent
Since it holds that 
$$h(G'')=(p+c_{p+\lambda-1}-1)+(q-c_{p+\lambda-1})
=p+q-1=h(G')-1,$$
by repeating this modification suitably, 
we finally obtain a Gauss diagram $G'''$ with $h(G''')=0$ 
which is S-equivalent to $G$. 

(ii) By (i), 
$G$ is S-equivalent to a Gauss diagram 
$$\left(\sum_{n\ne 0,1,-1}a_n S_1(n), 
\sum_{n\ne 0,1,2}b_n S_2(n);
c_0 S_{12}(p), d_0 S_{21}(-p) \right)$$

\noindent
for some integers $a_n$ $(n\ne 0,1,-1)$, 
$b_n$ $(n\ne 0,1,2)$, 
$c_0$, and $d_0$.
By Lemma~\ref{lem44}(ii) with $\lambda=1$, 
we can take $p=0$. 
\end{proof}

We remark that by Lemma~\ref{lem42}(ii) 
we have $c_0={\rm Lk}(K_1,K_2)$ and $d_0={\rm Lk}(K_2,K_1)$ 
in Proposition~\ref{prop46}(ii).

\begin{lemma}\label{lem47}
We have the following S-equivalent Gauss diagrams. 
\begin{itemize}
\setlength{\parskip}{2mm} 
 \setlength{\itemsep}{0cm} 
\item[{\rm (i)}]
If $\lambda=0$, then 

$\displaystyle{
\left(P, Q;\sum_{m\in{\Z}} c_m S_{12}(m), 
\sum_{m\in{\Z}} d_m S_{21}(m) \right)}$ 

\hfill
$\displaystyle{
\sim\left(P, Q;\sum_{m\in{\Z}} c_m S_{12}(m+k), 
\sum_{m\in{\Z}} d_m S_{21}(m-k) \right)}$

\noindent
for any $k\in{\Z}$. 

\item[{\rm (ii)}]
If $\lambda\geq 2$, then 

$\displaystyle{
\left(P, Q;\sum_{m=0}^{\lambda-1} c_m S_{12}(p+m), 
\sum_{m=0}^{\lambda-1} d_m S_{21}(-p-m) \right)}$ 

\hfill
$\displaystyle{
\sim
\left(P, Q;\sum_{m=0}^{\lambda-1} c_m' S_{12}(p'+m), 
\sum_{m=0}^{\lambda-1} d_m' S_{21}(-p'-m) \right),}$ 

where 
$$
\left\{
\begin{array}{l}
(c_0',\dots,c_{\lambda-k-1}', c_{\lambda-k}',\dots, c_{\lambda-1}')
=(c_k,\dots,c_{\lambda-1},c_0,\dots,c_{k-1}), \\
(d_0',\dots,d_{\lambda-k-1}',d_{\lambda-k}',\dots,d_{\lambda-1}')
=(d_k,\dots,d_{\lambda-1},d_0,\dots,d_{k-1}),
\end{array}\right.$$
and 
$p'=p+k-\sum_{i=0}^{k-1}(c_i-d_i)$ 
for any $k$ with $1\leq k\leq \lambda-1$. 
\end{itemize}
\end{lemma}

\begin{proof}
(i) 
Since $\lambda=0$, 
this follows from Lemma~\ref{lem44} 
immediately.

(ii) 
It is sufficient to prove the case of $k=1$. 
By Lemma~\ref{lem45}(i), 
the left hand side is S-equivalent to

$$\left(P, Q;
\sum_{m=1}^{\lambda-1} c_m S_{12}(p-c_0+m)
+ c_0 S_{12}(p-c_0+\lambda), 
\sum_{m=0}^{\lambda-1} 
d_m S_{21}(-p+c_0-m) \right).$$

\noindent
Furthermore, 
by Lemma~\ref{lem45}(iii), 
this is S-equivalent to

$\displaystyle{
\left(P, Q;
\sum_{m=1}^{\lambda-1} c_m S_{12}(p-c_0+d_0+m)
+ c_0 S_{12}(p-c_0+d_0+\lambda), \right.}$

\hfill 
$\displaystyle{
\left.\sum_{m=1}^{\lambda-1} 
d_m S_{21}(-p+c_0-d_0-m) 
+d_0S_{21}(-p+c_0-d_0-\lambda)\right)}$

$\displaystyle{
=\left(P, Q;
\sum_{m=0}^{\lambda-2} c_{m+1} S_{12}(p'+m)
+ c_0 S_{12}(p'+\lambda-1), \right.}$

\hfill 
$\displaystyle{
\left.\sum_{m=0}^{\lambda-2} 
d_{m+1} S_{21}(-p'-m) 
+d_0S_{21}(-p'-\lambda+1)\right),}$

\noindent
where $p'=p+1-c_0+d_0$. 
This is coincident with the right hand side in the case of $k=1$. 
\end{proof}

\section{The case $\mu=2$ (algebraic part)}\label{sec5}

Let $G$ be a Gauss diagram of 
an oriented $2$-component virtual link $L=K_1\cup K_2$. 
The index of a self- or nonself-chord $\gamma$ of $G$ 
is defined as follows. 
\begin{itemize}
\item[(i)] 
Let $\gamma$ be a self-chord spanning a circle $C_i$. 
The index of $\gamma$ in $G$ is the sum of signs 
of endpoints of self- and nonself-chords 
on the arc of $C_i$ 
oriented from the initial endpoint of $\gamma$  
to the terminal (cf.~\cite{Xu}). 
We denote it by ${\rm Ind}'(\gamma)$. 
\item[(ii)] 
Fix a nonself-chord $\gamma_0$ of $G$. 
Let $\gamma$ be a nonself-chord of type $(i,j)$. 
Let $\alpha$ be the arc on $C_i$ 
oriented from the initial endpoint of $\gamma$ 
to an endpoint of $\gamma_0$, 
and $\beta$ the arc on $C_j$ 
oriented from another endpoint of $\gamma_0$ to 
the terminal endpoint of $\gamma$. 
See the left of Figure~\ref{fig401}. 
The index of $\gamma$ with respect to $\gamma_0$ in $G$ 
is the sum of signs of endpoints of 
self- and nonself-chords on $\alpha\cup\beta$ 
(cf.~\cite{CG}). 
We denote it by ${\rm Ind}'(\gamma;\gamma_0)$. 
\end{itemize}

\begin{figure}[htb]
\begin{center}
\includegraphics[bb=0 0 211 60]{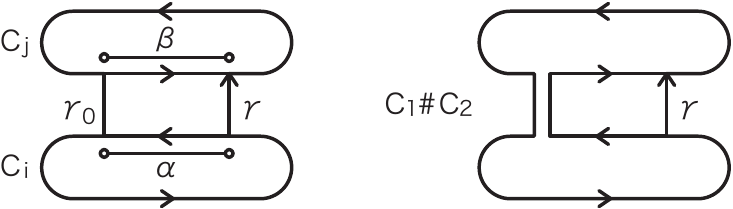}
\caption{The index of a nonself-chord}
\label{fig401}
\end{center}
\end{figure}

\begin{remark}\label{rem51}
We have two remarks. 
\begin{itemize}
\item[(i)] 
For a self-chord $\gamma$ spanning a circle $C_i$, 
the index ${\rm Ind}'(\gamma)$ is generally not equal to 
the original index ${\rm Ind}(\gamma)$ in Section~\ref{sec3} 
restricted to the Gauss diagram consisting of the circle $C_i$ 
with self-chords spanning $C_i$. 
\item[(ii)] 
For a nonself-chord $\gamma$ of type $(i,j)$, 
the index ${\rm Ind}'(\gamma;\gamma_0)$ is 
equal to the index ${\rm Ind}(\gamma)$ 
in the Gauss diagram consisting of 
the circle $C_1\# C_2$ 
obtained from $G$ by surgery along $\gamma_0$. 
See the right of Figure~\ref{fig401}. 
\end{itemize}
\end{remark}

\begin{lemma}\label{lem52}
Let $\gamma$ be a chord of $G$. 
\begin{itemize}
\item[{\rm (i)}] 
If $\gamma$ is a shell spanning $C_1$, 
then ${\rm Ind}'(\gamma)=1,-\lambda+1$. 
\item[{\rm (ii)}] 
If $\gamma$ is a shell spanning $C_2$, 
then ${\rm Ind}'(\gamma)=1, \lambda+1$. 
\item[{\rm (iii)}] 
If $\gamma$ is not a shell, 
then the index does not change under S-moves. 
\end{itemize} 
\end{lemma}

\begin{proof}
This follows from the definition of the index 
immediately. 
\end{proof}

For an integer $n$, 
we denote by $J_n^i(G)$ $(i=1,2)$ 
the sum of signs of all self-chords $\gamma$ spanning $C_i$ 
with ${\rm Ind}'(\gamma)=n$. 
It is known in \cite{Xu} that 
$J_n^1(G)$ is independent of a particular choice 
of $G$ for $n\ne 0,-\lambda$. 
It is called the {\it $n$-writhe of $K_1$ in $L$} 
and denoted by $J_n(K_1;L)$ for $n\ne 0,-\lambda$. 
Similarly $J_n^2(G)$ is independent of 
a particular choice of $G$ for $n\ne 0,\lambda$. 
It is called the {\it $n$-writhe of $K_2$ in $L$} 
and denoted by $J_n(K_2;L)$ for $n\ne 0,\lambda$. 
We remark that the index of a free chord spanning $C_1$ 
(or $C_2$) is equal to $0$ or $-\lambda$ 
(or $0$ or $\lambda$).

\begin{example}\label{ex53}
We consider the Gauss diagram 
$$G=(2S_1(2)-S_1(3),2S_2(-1); 
S_{12}(0)+S_{12}(-1)+S_{12}(4), 
2S_{21}(2)-S_{21}(3))$$
as shown in Figure~\ref{fig502}. 
Let $L=K_1\cup K_2$ be the oriented $2$-component virtual link 
presented by $G$. 
We have 
$${\rm Lk}(K_1,K_2)=3, \ 
{\rm Lk}(K_2,K_1)=1, 
\mbox{ and }\lambda(L)=2.$$
Furthermore, it holds that 
$$J_n(K_1;L)=
\left\{
\begin{array}{rl}
-5 & (n=-1), \\
2 & (n=2), \\
-1 & (n=3), \\
0 & (\mbox{otherwise}), 
\end{array}\right.
\mbox{ and } 
J_n(K_2;L)=
\left\{
\begin{array}{rl}
2 & (n=-1), \\
1 & (n=1,3), \\
0 & (\mbox{otherwise}). 
\end{array}\right.$$
\end{example}

\begin{figure}[htb]
\begin{center}
\includegraphics[bb=0 0 308 164]{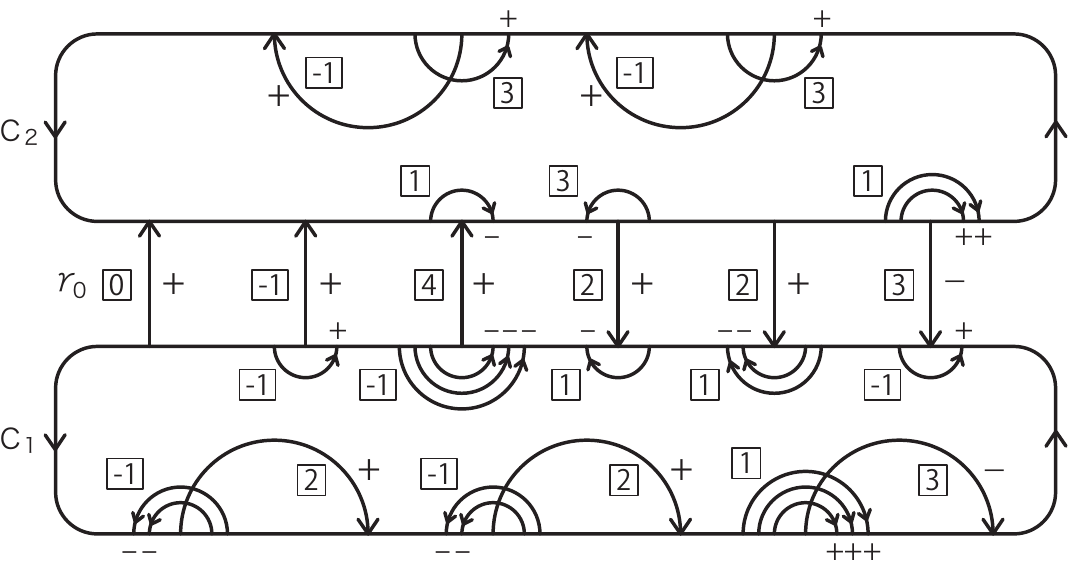}
\caption{The Gauss diagram in Example~\ref{ex53}}
\label{fig502}
\end{center}
\end{figure}

\begin{lemma}\label{lem54}
Let $L=K_1\cup K_2$ be an oriented $2$-component virtual link. 
\begin{itemize}
\item[{\rm (i)}] 
The $n$-writhes 
\begin{itemize}
\item[$\bullet$] 
$J_n(K_1;L)\in{\Z}$ 
$(n\ne 0,1,-\lambda,-\lambda+1)$ and 
\item[$\bullet$] 
$J_n(K_2;L)\in{\Z}$ 
$(n\ne 0,1,\lambda,\lambda+1)$ 
\end{itemize}
are invariant under S-moves. 
\item[{\rm (ii)}] 
If $L$ is presented by a Gauss diagram 
given in {\rm Lemma~\ref{lem41}}, 
then we have 
\begin{itemize}
\item[$\bullet$]
$J_n(K_1;L)=a_n$ $(n\ne 0,1,-\lambda,-\lambda+1)$ and 
\item[$\bullet$]
$J_n(K_2;L)=b_n$ 
$(n\ne 0,1,\lambda,\lambda+1)$.
\end{itemize}
\end{itemize}
\end{lemma}

\begin{proof}
(i) 
This follows from Lemma~\ref{lem52}. 

(ii) 
For $n\ne 0,1,-\lambda,-\lambda+1$, 
only the union of snails $a_nS_1(n)$ 
contributes to $J_n(K_1;L)$. 
Similarly, 
for $n\ne 0,1,\lambda,\lambda+1$, 
only the union of snails $b_nS_2(n)$ 
contributes to $J_n(K_2;L)$. 
\end{proof}

\begin{lemma}\label{lem55}
Let $L=K_1\cup K_2$ be an oriented $2$-component virtual link. 
\begin{itemize}
\item[{\rm (i)}] 
The sums of writhes 
\begin{itemize}
\item[$\bullet$] 
$J_1(K_1;L)+J_1(K_2;L)\in{\Z}$ for $\lambda=0$ and 
\item[$\bullet$] 
$J_1(K_1;L)+J_{-\lambda+1}(K_1;L)
+J_1(K_2;L)+J_{\lambda+1}(K_2;L)\in{\Z}$ for $\lambda\geq 2$ 
\end{itemize}
are invariant under S-moves. 
\item[{\rm (ii)}] 
If $L$ is presented by a Gauss diagram 
given in {\rm Lemma~\ref{lem41}}, 
then the invariants in {\rm (i)} are given by 
$$-\sum_{n\ne 0,1,-\lambda,-\lambda+1}na_n
-\sum_{n\ne 0,1,\lambda,\lambda+1}nb_n
-\sum_{m\in{\Z}}mc_m
-\sum_{m\in{\Z}}md_m$$
for $\lambda=0$ and $\lambda\geq 2$. 
\end{itemize}
\end{lemma}

\begin{proof}
(i) 
For the case of $\lambda=0$, 
any shell has the index $1$ by Lemma~\ref{lem52}. 
By an S1-move, 
a shell spanning $C_1$ (or $C_2$) 
may change into the one spanning $C_2$ (or $C_1$). 
Since the sign of the shell does not change, 
$J_1(K_1;L)+J_1(K_2;L)$ is invariant under an S1-move. 
On the other hand, 
since the produced (or canceled) pair of shells 
by an S2-move have the opposite signs, 
$J_1(K_1;L)+J_1(K_2;L)$ is also invariant under an S2-move.

For the case of $\lambda\geq 2$, 
there are four types of shells as shown in Lemma~\ref{lem52}. 
By an S1-move, 
a shell may change into the one as follows: 

\begin{center}
\begin{tabular}{ccc}
spanning $C_1$ & & spanning $C_2$ \\
\hline
a shell of index $1$ & $\leftrightarrow$ & a shell of index $\lambda+1$\\
$\updownarrow$ & & $\updownarrow$ \\
a shell of index $-\lambda+1$ & $\leftrightarrow$ & 
a shell of index $1$
\end{tabular}
\end{center}
Therefore 
$J_1(K_1;L)+J_{-\lambda+1}(K_1;L)
+J_1(K_2;L)+J_{\lambda+1}(K_2;L)$ 
is invariant under an S1-move. 
The invariance under an S2-move 
is proved similarly to the case of $\lambda=0$. 

(ii) 
For the case of $\lambda=0$, 
the union of snails 
$a_nS_1(n)$ contributes $-na_n$ to $J_1(K_1;L)$, 
and $b_nS_2(n)$ contributes $-nb_n$ to $J_1(K_2;L)$. 
Furthermore, 
$c_mS_{12}(m)$ and $d_mS_{21}(m)$ 
contributes $-mc_m$ and $-md_m$ to 
$J_1(K_1;L)+J_1(K_2;L)$, 
respectively. 
The case of $\lambda\geq 2$ can be similarly proved. 
\end{proof}

For $n\in{\Z}$, $(i,j)\in\{(1,2),(2,1)\}$, and 
a nonself-chord $\gamma_0$, 
we denote by $J_n^{ij}(G;\gamma_0)$ 
the sum of signs of nonself-chords $\gamma$ 
of type $(i,j)$ with ${\rm Ind}'(\gamma;\gamma_0)=n$. 
Put 
$$F_{ij}(t;\gamma_0)=\sum_{n\in{\Z}}
J_n^{ij}(G;\gamma_0)t^n.$$
We remark that 
$F_{ij}(1;\gamma_0)={\rm Lk}(K_i,K_j)$ 
holds by definition. 
For any nonself-chords $\gamma_0$ and $\gamma_1$, 
there is an integer $k$ such that 
$$F_{12}(t;\gamma_1)=t^k F_{12}(t;\gamma_0) 
\mbox{ and }
F_{21}(t;\gamma_1)=t^{-k}F_{21}(t;\gamma_0).$$

For an integer $s\geq 0$, 
let $\Lambda_s$ denote the Laurent polynomial ring 
${\Z}[t,t^{-1}]/(t^s-1)$. 
In particular, we have $\Lambda_0={\Z}[t,t^{-1}]$ 
and $\Lambda_1={\Z}$. 
We consider an equivalence relation 
on $\Lambda_s\times\Lambda_s$ such that 
$\left(f_1(t),g_1(t)\right)$ and $\left(f_2(t),g_2(t)\right)$
are equivalent 
if there is an integer $k$ with 
$$f_2(t)=t^k f_1(t) \mbox{ and } 
g_2(t)=t^{-k}g_1(t).$$
We denote by 
$\left[f(t),g(t)\right]$ 
the equivalence class represented by 
$\left(f(t),g(t)\right)$, 
and by $\Gamma(s)$ 
the set of such equivalence classes. 
By definition, we have 
$\Gamma(1)={\Z}\times {\Z}$.

It is known in \cite{CG} that the equivalence class 
 $$\left[F_{12}(t;\gamma_0),F_{21}(t;\gamma_0)\right]\in
\Gamma(\lambda)$$
is independent of a particular choice of 
$\gamma_0$ and $G$ for $L$. 
It is called the {\it linking class} of $L$ 
and denoted by 
$F(L)\in\Gamma(\lambda)$. 
In particular, if $\lambda=1$, 
then $F(L)$ is identified with 
the pair 
$({\rm Lk}(K_1,K_2),{\rm  Lk}(K_2,K_1))
\in{\Z}\times {\Z}$. 

\begin{example}\label{ex56} 
We consider the Gauss diagram $G$ and 
the oriented $2$-component virtual link $L$ 
given in Example~\ref{ex53}. 
Let $\gamma_0$ be the nonself-chord of $G$ 
as shown in Figure~\ref{fig502}. 
Then it holds that 
$$J_n^{12}(G;\gamma_0)=
\left\{
\begin{array}{rl}
1 & (n=-1,0,4), \\
0 & (\mbox{otherwise}), 
\end{array}\right.
\mbox{ and } 
J_n^{21}(G;\gamma_0)=
\left\{
\begin{array}{rl}
2 & (n=2), \\
-1 & (n=3), \\
0 & (\mbox{otherwise}). 
\end{array}\right.$$
Therefore, we have 
$F(L)=[t^{-1}+1+t^4,2t^2+t^3]
\in\Gamma(2)$. 
\end{example}

\begin{lemma}\label{lem57}
Let $L=K_1\cup K_2$ be an oriented $2$-component virtual link. 
\begin{itemize}
\item[{\rm (i)}] 
The linking class 
$F(L)\in\Gamma(\lambda)$ is invariant under S-moves. 
\item[{\rm (ii)}] 
If $L$ is presented by a Gauss diagram 
given in {\rm Lemma~\ref{lem41}}, 
then we have 
$\displaystyle{F(L)=\left[\sum_{m\in{\Z}}c_mt^m, 
\sum_{m\in{\Z}}d_mt^m\right]}$. 
\end{itemize}
\end{lemma}

\begin{proof}
(i) 
Since any nonself-chord is not a shell, 
we have the invariance of $F(L)$ by 
Lemma~\ref{lem52}(iii). 

(ii) 
First we add a pair of nonself-chords 
$+S_{12}(0)$ and $-S_{12}(0)$ by an R2-move. 
Put $\gamma_0=+S_{12}(0)$. 
Then the union of snails $c_mS_{12}(m)$ 
contributes $c_m$ to $J_m^{12}(G;\gamma_0)$, 
and $d_mS_{21}(m)$ contributes $d_m$ 
to $J_m^{21}(G;\gamma_0)$. 
Therefore we have the equation by the definition of $F(L)$. 
\end{proof}

\begin{theorem}\label{thm58}
Let $L=K_1\cup K_2$ and $L'=K_1'\cup K_2'$ 
be oriented $2$-component virtual links with $\lambda=\lambda'=0$. 
Suppose that 
\begin{itemize} 
\item[{\rm (i)}] 
$J_n(K_1;L)=J_n(K_1';L')$ for any $n\ne 0, 1$, 
\item[{\rm (ii)}] 
$J_n(K_2;L)=J_n(K_2';L')$ for any $n\ne 0, 1$, and 
\item[{\rm (iii)}] 
$F(L)=F(L')$. 
\end{itemize}
Then $L$ and $L'$ are S-equivalent. 
\end{theorem}

\begin{proof}
By Lemma~\ref{lem41}, 
any Gauss diagrams of $L$ and $L'$ are S-equivalent to 
Gauss diagrams 
$$
\begin{array}{l} 
\displaystyle{
G=\left(\sum_{n\ne 0,1}a_n S_1(n), 
\sum_{n\ne 0,1}b_n S_2(n); 
\sum_{m\in{\Z}}c_m S_{12}(m), 
\sum_{m\in{\Z}}d_m S_{21}(m) \right) \mbox{ and }}\\
\displaystyle{
G'=\left(\sum_{n\ne 0,1}a_n' S_1(n), 
\sum_{n\ne 0,1}b_n' S_2(n); 
\sum_{m\in{\Z}}c_m' S_{12}(m), 
\sum_{m\in{\Z}}d_m' S_{21}(m) \right),}
\end{array}$$
respectively. 
By Lemma~\ref{lem54}(ii) and the assumption, 
we obtain $a_n=a_n'$ and $b_n=b_n'$ for any $n\ne 0,1$.

Furthermore, 
by Lemma~\ref{lem57}(ii) and the assumption, 
we obtain 
$$\left[\sum_{m\in{\Z}}c_mt^m, 
\sum_{m\in{\Z}}d_mt^m\right]
=\left[\sum_{m\in{\Z}}c_m't^m, 
\sum_{m\in{\Z}}d_m't^m\right].$$
Then there is an integer $k$ such that 
$$\sum_{m\in{\Z}} c_m't^m=t^k\sum_{m\in{\Z}} c_mt^m
 \mbox{ and }
 \sum_{m\in{\Z}} d_m't^m=t^{-k}\sum_{m\in{\Z}} d_mt^m$$
so that we obtain 
$c_m'=c_{m-k}$ and $d_m'=d_{m+k}$ for any $m\in{\Z}$. 
Therefore it holds that 
\begin{eqnarray*}
G'&=&
\left(\sum_{n\ne 0,1}a_n S_1(n), 
\sum_{n\ne 0,1}b_n S_2(n); 
\sum_{m\in{\Z}}c_{m-k} S_{12}(m), 
\sum_{m\in{\Z}}d_{m+k} S_{21}(m) \right)\\
&=&
\left(\sum_{n\ne 0,1}a_n' S_1(n), 
\sum_{n\ne 0,1}b_n' S_2(n); 
\sum_{m\in{\Z}}c_m' S_{12}(m+k), 
\sum_{m\in{\Z}}d_m' S_{21}(m-k) \right)\\
&\sim& G
\end{eqnarray*}
by Lemma~\ref{lem47}(i). 
\end{proof}

\begin{theorem}\label{thm59}
Let $L=K_1\cup K_2$ and $L'=K_1'\cup K_2'$ 
be oriented $2$-component virtual links with $\lambda=\lambda'=1$. 
Suppose that 
\begin{itemize} 
\item[{\rm (i)}] 
$J_n(K_1;L)=J_n(K_1';L')$ for any $n\ne 0, 1,-1$, 
\item[{\rm (ii)}] 
$J_n(K_2;L)=J_n(K_2';L')$ for any $n\ne 0, 1,2$, and 
\item[{\rm (iii)}] 
$F(L)=F(L')$. 
\end{itemize}
Then $L$ and $L'$ are S-equivalent. 
\end{theorem}

\begin{proof}
By Proposition~\ref{prop46}(ii), 
any Gauss diagrams of $L$ and $L'$ are S-equivalent to 
Gauss diagrams 
$$
\begin{array}{l} 
\displaystyle{
G=\left(\sum_{n\ne 0,1,-1}a_n S_1(n), 
\sum_{n\ne 0,1,2}b_n S_2(n); 
c_0 S_{12}(0), 
d_0 S_{21}(0) \right) \mbox{ and }}\\
\displaystyle{
G'=\left(\sum_{n\ne 0,1,-1}a_n' S_1(n), 
\sum_{n\ne 0,1,2}b_n' S_2(n); 
c_0' S_{12}(0), 
d_0' S_{21}(0) \right),}
\end{array}$$
respectively. 
By Lemma~\ref{lem54}(ii) and the assumption, 
we obtain $a_n=a_n'$ $(n\ne 0,1,-1)$ and 
$b_n=b_n'$ $(n\ne 0,1,2)$. 

Furthermore, 
since $F(L)=(c_0,d_0)$ and 
$F(L')=(c_0',d_0')\in{\Z}\times{\Z}$, 
we have $c_0=c_0'$ and $d_0=d_0'$ 
by the assumption. 
Therefore $G=G'$ holds. 
\end{proof}

\begin{theorem}\label{thm510}
Let $L=K_1\cup K_2$ and $L'=K_1'\cup K_2'$ 
be oriented $2$-component virtual links with $\lambda=\lambda'\geq 2$. 
Suppose that 
\begin{itemize} 
\item[{\rm (i)}] 
$J_n(K_1;L)=J_n(K_1';L')$ for any $n\ne 0, 1,-\lambda,-\lambda+1$, 
\item[{\rm (ii)}] 
$J_n(K_2;L)=J_n(K_2';L')$ for any $n\ne 0, 1,\lambda,\lambda+1$, 
\item[{\rm (iii)}] 
$F(L)=F(L')$, and 
\item[{\rm (iv)}] 
$J_1(K_1;L)+J_{-\lambda+1}(K_1;L)
+J_1(K_2;L)+J_{\lambda+1}(K_2;L)$

\hfill
$=J_1(K_1';L')+J_{-\lambda+1}(K_1';L')
+J_1(K_2';L')+J_{\lambda+1}(K_2';L')$.
\end{itemize}
Then $L$ and $L'$ are S-equivalent. 
\end{theorem}

\begin{proof}
By Proposition~\ref{prop46}(i), 
any Gauss diagrams of $L$ and $L'$ are S-equivalent to 
Gauss diagrams 

$\displaystyle{
G=\left(\sum_{n\ne 0,1,-\lambda,-\lambda+1}a_n S_1(n), 
\sum_{n\ne 0,1,\lambda,\lambda+1}b_n S_2(n);\right.}$

\hfill
$\displaystyle{
\left.\sum_{m=0}^{\lambda-1}c_m S_{12}(p+m), 
\sum_{m=0}^{\lambda-1}d_m S_{21}(-p-m) \right)}$

\noindent
and 

$\displaystyle{
G'=\left(\sum_{n\ne 0,1,-\lambda,-\lambda+1}a_n' S_1(n), 
\sum_{n\ne 0,1,\lambda,\lambda+1}b_n' S_2(n);\right.}$

\hfill
$\displaystyle{
\left.\sum_{m=0}^{\lambda-1}c_m' S_{12}(p'+m), 
\sum_{m=0}^{\lambda-1}d_m' S_{21}(-p'-m) \right),}$

\noindent
respectively. 
By Lemma~\ref{lem54}(ii) and the assumption, 
we obtain $a_n=a_n'$ 
for any $n\ne 0,1,-\lambda,-\lambda+1$ 
and $b_n=b_n'$ for any $n\ne 0,1,\lambda,\lambda+1$. 

Next, 
by Lemma~\ref{lem57}(ii) and the assumption, 
we obtain 
$$\left[\sum_{m=0}^{\lambda-1}c_mt^{p+m}, 
\sum_{m=0}^{\lambda-1}d_mt^{-p-m}\right]
=\left[\sum_{m=0}^{\lambda-1}c_m't^{p'+m}, 
\sum_{m=0}^{\lambda-1}d_m't^{-p'-m}\right]
\in\Gamma(\lambda),$$
or equivalently, 
$$\left[\sum_{m=0}^{\lambda-1}c_mt^m, 
\sum_{m=0}^{\lambda-1}d_mt^{-m}\right]
=\left[\sum_{m=0}^{\lambda-1}c_m't^m, 
\sum_{m=0}^{\lambda-1}d_m't^{-m}\right]
\in\Gamma(\lambda).$$
Then there is an integer $k$ with $1\leq k\leq \lambda-1$ 
such that 

$$
\left\{
\begin{array}{l}
(c_0',\dots,c_{\lambda-k-1}', c_{\lambda-k}',\dots, c_{\lambda-1}')
=(c_k,\dots,c_{\lambda-1},c_0,\dots,c_{k-1}), \\
(d_0',\dots,d_{\lambda-k-1}',d_{\lambda-k}',\dots,d_{\lambda-1}')
=(d_k,\dots,d_{\lambda-1},d_0,\dots,d_{k-1}). 
\end{array}\right.$$

Furthermore, 
by Lemma~\ref{lem55}(ii) and the assumption, 
it holds that 

$$\sum_{m=0}^{\lambda-1}(p+m)c_m
-\sum_{m=0}^{\lambda-1}(p+m)d_m
=
\sum_{m=0}^{\lambda-1}(p'+m)c_m'
-\sum_{m=0}^{\lambda-1}(p'+m)d_m'.$$
By Lemma~\ref{lem42}(ii), 
this is equivalent to 
$$(p-p')\lambda = 
\left(\sum_{m=0}^{\lambda-1}mc_m'
-\sum_{m=0}^{\lambda-1}mc_m\right)
-\left(\sum_{m=0}^{\lambda-1}md_m'
-\sum_{m=0}^{\lambda-1}md_m\right).$$

\noindent
Here, it holds that 
\begin{eqnarray*}
\sum_{m=0}^{\lambda-1}mc_m'
&=&
\sum_{m=k}^{\lambda-1}(m-k)c_m
+\sum_{m=0}^{k-1}(m+\lambda-k)c_m\\
&=&
\sum_{m=0}^{\lambda-1}mc_m
-k \sum_{m=0}^{\lambda-1}c_m
+\lambda\sum_{m=0}^{k-1} c_m.
\end{eqnarray*}
Since we have a similar equation for $d_m'$, 
it holds that 
$$
(p-p')\lambda=-k\lambda+\lambda
\sum_{m=0}^{k-1}(c_m-d_m),$$
that is, 
$p'=p+k-\sum_{m=0}^{k-1}(c_m-d_m)$. 
Therefore, $G'$ is S-equivalent to $G$ 
by Lemma~\ref{lem47}(ii). 
\end{proof}

\begin{proof}[Proofs of {\rm Theorems~\ref{thm12}, \ref{thm13}}, 
and {\rm \ref{thm14}}.] 
This follows from 
Lemmas~\ref{lem54}(i), \ref{lem55}(i), \ref{lem57}(i), 
Theorems~\ref{thm58}, \ref{thm59}, and \ref{thm510}
immediately. 
\end{proof}

\section{A relation among invariants}\label{sec6}

In this section, 
we study a relationship among invariants 
which are used in the previous section.

\begin{lemma}\label{lem61}
Let $s$ be a nonnegative integer with $s\ne 1$, 
and $f_i(t)$ and $g_i(t)$ $(i=1,2)$ 
Laurent polynomials in ${\Z}[t,t^{-1}]$. 
Suppose that 
\begin{itemize}
\item[{\rm (i)}] 
$[f_1(t),g_1(t)]=[f_2(t),g_2(t)]\in \Gamma(s)$ and 
\item[{\rm (ii)}] 
$f_1(1)-g_1(1)=f_2(1)-g_2(1)=s$. 
\end{itemize}
Then it holds that 
$$f_1'(1)+g_1'(1)\equiv 
f_2'(1)+g_2'(1) \ ({\rm mod}\ s).$$
In particular, if $s=0$, 
then 
$f_1'(1)+g_1'(1)=
f_2'(1)+g_2'(1)\in{\Z}$. 
\end{lemma}

\begin{proof}
By the condition (i), 
there are $k\in{\Z}$ and 
$\varphi(t),\psi(t)\in {\Z}[t,t^{-1}]$ such that 
$$f_2(t)=t^k f_1(t)+(t^s-1)\varphi(t) 
\mbox{ and }
g_2(t)=t^{-k} g_1(t)+(t^s-1)\psi(t).$$
Then we have 
$$
\left\{
\begin{array}{l}
f_2'(1)=kf_1(1)+f_1'(1)+s\varphi(1) \mbox{ and }\\
g_2'(1)=-k g_1(1)+g_1'(1)+s\psi(1).
\end{array}\right.$$
Therefore it holds that 
$$f_2'(1)+g_2'(1)=
f_1'(1)+g_1'(1)+s(k+\varphi(1)+\psi(1)),$$
and we have the conclusion. 
\end{proof}

For an oriented $2$-component virtual link $L$, 
the linking class 
$F(L)=[f(t),g(t)]\in\Gamma(\lambda)$ 
satisfies 
$$f(1)-g(1)={\rm Lk}(K_1,K_2)-{\rm Lk}(K_2,K_1)=\lambda.$$
Therefore 
$f'(1)+g'(1)$ 
$({\rm mod}\ \lambda)$ is well-defined, 
and denoted by $F'(L)\in{\Z}/\lambda{\Z}$. 
We remark that, since $F(L)$ is invariant 
under S-moves, 
so is $F'(L)$.

\begin{proposition}\label{prop62}
Let $L=K_1\cup K_2$ 
be an oriented $2$-component virtual link. 

\begin{itemize}
\item[{\rm (i)}] 
If $\lambda=0$, then 
$$\sum_{n\ne 0}nJ_n(K_1;L)+
\sum_{n\ne 0}n J_n(K_2;L)+F'(L)=0.$$

\item[{\rm (ii)}]
If $\lambda\geq 2$, then 
$$\sum_{n\ne 0,-\lambda}nJ_n(K_1;L)+
\sum_{n\ne 0,\lambda}n J_n(K_2;L)+F'(L)
\equiv 0 \ ({\rm mod}~\lambda).$$
\end{itemize}
\end{proposition}

\begin{proof}
(i) 
By Lemmas~\ref{lem54}(i), \ref{lem55}(i), and \ref{lem57}(i), 
the left hand side of the congruence 
is invariant under S-moves. 
Therefore, 
it is sufficient to consider a Gauss diagram 
given in Lemma~\ref{lem41}.
By Lemmas~\ref{lem54}(ii), \ref{lem55}(ii), and \ref{lem57}(ii), 
we have the conclusion. 

(ii) The proof is similar to that of (i). 
\end{proof}

Recall that in the case of $\lambda=0$, 
an oriented $2$-component virtual link $L$ has the invariants 
$J_n(K_1;L)$, $J_n(K_2;L)$ $(n\ne 0)$, 
and $F(L)\in\Gamma(0)$. 

\begin{theorem}\label{thm63}
Let $a_n$, $b_n$ $(n\ne 0)$, and 
$c_m$, $d_m$ $(m\in{\Z})$ 
be integers such that 
\begin{itemize}
\setlength{\parskip}{1mm} 
 \setlength{\itemsep}{0cm} 
\item[{\rm (a)}] 
$\displaystyle{
\sum_{m\in{\Z}}c_m=\sum_{m\in{\Z}}d_m}$ and 
\item[{\rm (b)}] 
$\displaystyle{
\sum_{n\ne 0}na_n+\sum_{n\ne 0}nb_n 
+\sum_{m\in{\Z}}mc_m+\sum_{m\in{\Z}}md_m=0}$. 
\end{itemize}
Then there is an oriented $2$-component 
virtual link $L=K_1\cup K_2$ such that 
\begin{itemize}
\item[{\rm (i)}] 
$J_n(K_1;L)=a_n$ $(n\ne 0)$, 
\item[{\rm (ii)}] 
$J_n(K_2;L)=b_n$ $(n\ne 0)$, and 
\item[{\rm (iii)}] 
$\displaystyle{
F(L)=\left[
\sum_{m\in{\Z}}c_mt^m, 
\sum_{m\in{\Z}}d_mt^m\right]\in\Gamma(0)}$. 
\end{itemize}
\end{theorem}

\begin{proof}
Let $G$ be the Gauss diagram 
$$
\left(\sum_{n\ne 0,1}a_n S_1(n), 
\sum_{n\ne 0,1}b_n S_2(n); 
\sum_{m\in{\Z}}c_m S_{12}(m), 
\sum_{m\in{\Z}}d_m S_{21}(m) \right).$$
The virtual link $L$ presented by $G$ 
satisfies (i)--(iii) 
except 
\begin{eqnarray*}
J_1(K_1;L)+J_1(K_2;L)
&=&
-\sum_{n\ne 0,1}na_n-\sum_{n\ne 0,1}nb_n
-\sum_{m\in{\Z}}mc_m-\sum_{m\in{\Z}}md_m\\
&=&
a_1+b_1
\end{eqnarray*}
by Lemma~\ref{lem55}(ii) and 
the condition (b). 

Put $x=a_1-J_1(K_1;L)$. 
Fix a nonself-chord $\gamma$ 
in the snails of $\sum_{m\in{\Z}}c_m S_{12}(m)$ and 
$\sum_{m\in{\Z}}d_m S_{21}(m)$. 
Let $G'$ be the Gauss diagram obtained from $G$ 
by adding 
\begin{itemize}
\item
shells spanning $C_1$ 
such that the sum of signs are equal to $x$, and 
\item
shells spanning $C_2$ 
such that the sum of signs are equal to $-x$
\end{itemize}
to $\gamma$. 
Then the virtual link $L'$ presented by $G'$ 
has the same invariants as $L$ except 
$$\left\{
\begin{array}{l}
J_1(K_1';L')=J_1(K_1;L)+x=a_1 \mbox{ and}\\
J_1(K_2';L')=J_1(K_2;L)-x=b_1. 
\end{array}\right.$$
This virtual link $L'$ is a desired one. 
\end{proof}

Recall that in the case of $\lambda=1$, 
an oriented $2$-component virtual link $L$ has the invariants 
$J_n(K_1;L)$ $(n\ne 0,-1)$, 
$J_n(K_2;L)$ $(n\ne 0,1)$, 
and $$F(L)=({\rm Lk}(K_1,K_2), 
{\rm Lk}(K_2,K_1))\in\Gamma(1)={\Z}\times{\Z}.$$

\begin{theorem}\label{thm64}
Let $a_n$ $(n\ne 0,-1)$, $b_n$ $(n\ne 0,1)$, and 
$c$ be integers. 
Then there is an oriented $2$-component 
virtual link $L=K_1\cup K_2$ such that 
\begin{itemize}
\item[{\rm (i)}] 
$J_n(K_1;L)=a_n$ $(n\ne 0,-1)$, 
\item[{\rm (ii)}] 
$J_n(K_2;L)=b_n$ $(n\ne 0,1)$, and 
\item[{\rm (iii)}] 
$F(L)=(c,c-1)$. 
\end{itemize}
\end{theorem}

\begin{proof}
Let $G$ be the Gauss diagram 
$$
\left(\sum_{n\ne 0,1,-1}a_n S_1(n), 
\sum_{n\ne 0,1,2}b_n S_2(n); 
c S_{12}(0), 
(c-1) S_{21}(0) \right).$$
The virtual link $L$ presented by $G$ 
satisfies (i)--(iii) except 
$J_1(K_1;L)$ and $J_2(K_2;L)$. 

Consider four kinds of portions of self-chords 
such that two of them span $C_1$ and 
the other two span $C_2$ 
as shown in Figure~\ref{fig601}. 
Adding these portions to $G$ suitably, 
we can change $J_1(K_1;L)$ and $J_2(K_2;L)$ 
arbitrarily with keeping other invariants 
so that we realize $a_1$ and $b_2$, respectively. 
We remark that $J_0(K_1;L)$ and $J_1(K_2;L)$ 
are not defined in the case of $\lambda=1$. 
\end{proof}

\begin{figure}[htb]
\begin{center}
\includegraphics[bb=0 0 312 27]{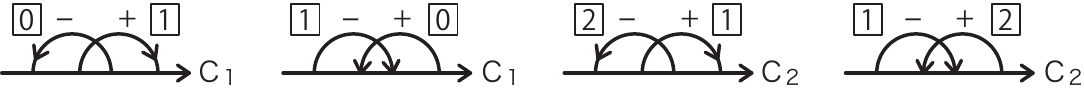}
\caption{Changing $J_1(K_1;L)$ and $J_2(K_2;L)$}
\label{fig601}
\end{center}
\end{figure}

Recall that in the case of $\lambda\geq 2$, 
an oriented $2$-component virtual link $L$ 
has the invariants $J_n(K_1;L)$ $(n\ne 0,-\lambda)$, 
$J_n(K_2;L)$ $(n\ne 0,\lambda)$, 
and $F(L)\in\Gamma(\lambda)$.

\begin{theorem}\label{thm65}
Let $\lambda\geq 2$, 
$a_n$ $(n\ne 0,-\lambda)$, $b_n$ $(n\ne 0,\lambda)$,  
and $c_m$, $d_m$ $(0\leq m\leq \lambda-1)$ 
be integers such that 
\begin{itemize}
\setlength{\parskip}{1mm} 
 \setlength{\itemsep}{0cm} 
\item[{\rm (a)}] 
$\displaystyle{
\sum_{m=0}^{\lambda-1}c_m
-\sum_{m=0}^{\lambda-1}d_m=\lambda}$ and 
\item[{\rm (b)}] 
$\displaystyle{
\sum_{n\ne 0,-\lambda}na_n+\sum_{n\ne 0,\lambda}nb_n 
+\sum_{m\in{\Z}}mc_m-\sum_{m\in{\Z}}md_m
\equiv 0}$ $({\rm mod}~\lambda)$. 
\end{itemize}
Then there is an oriented $2$-component 
virtual link $L=K_1\cup K_2$ such that 
\begin{itemize}
\item[{\rm (i)}] 
$J_n(K_1;L)=a_n$ $(n\ne 0,-\lambda)$, 
\item[{\rm (ii)}] 
$J_n(K_2;L)=b_n$ $(n\ne 0,\lambda)$, and 
\item[{\rm (iii)}] 
$\displaystyle{
F(L)=\left[
\sum_{m=0}^{\lambda-1}c_mt^m, 
\sum_{m=0}^{\lambda-1}d_mt^{-m}\right]\in\Gamma(\lambda)}$. 
\end{itemize}
\end{theorem}

\begin{proof}
There is an integer $k$ such that 
the left hand side of (b) is equal to $k\lambda$. 
Let $G$ be the Gauss diagram 

$\displaystyle{
G=\left(\sum_{n\ne 0,1,-\lambda,-\lambda+1}a_n S_1(n), 
\sum_{n\ne 0,1,\lambda,\lambda+1}b_n S_2(n);\right.}$

\hfill
$\displaystyle{
\left.\sum_{m=0}^{\lambda-1}c_m S_{12}(p+m), 
\sum_{m=0}^{\lambda-1}d_m S_{21}(-p-m) \right),}$

\noindent
where $p=-k-a_{-\lambda+1}+b_{\lambda+1}$. 
The virtual link $L$ presented by $G$ 
satisfies (i)--(iii) 
except 
\begin{eqnarray*}
&&J_1(K_1;L)+J_{-\lambda+1}(K_1;L)
+J_1(K_2;L)+J_{\lambda+1}(K_2;L) \\
&&=
-\sum_{n\ne 0,1,-\lambda,-\lambda+1}na_n
-\sum_{n\ne 0,1,\lambda,\lambda+1}nb_n
-\sum_{m\in{\Z}}(p+m)c_m+\sum_{m\in{\Z}}(p+m)d_m\\
&&=
a_1+(-\lambda+1)a_{-\lambda+1}+b_1
+(\lambda+1)b_{\lambda+1}
-k\lambda-p\left(
\sum_{m=0}^{\lambda-1}c_m
-\sum_{m=0}^{\lambda-1}d_m\right)\\
&&=
a_1+a_{-\lambda+1}+b_1+b_{\lambda+1}. 
\end{eqnarray*}
by Lemma~\ref{lem55}(ii) 
and the conditions (a) and (b).

Put $x=a_1+a_{-\lambda+1}-J_1(K_1;L)
-J_{-\lambda+1}(K_1;L)$. 
Fix a nonself-chord $\gamma$ 
in the snails of 
$\sum_{m=0}^{\lambda-1}c_mS_{12}(p+m)$ 
and $\sum_{m=0}^{\lambda-1}d_mS_{21}(-p-m)$. 
Let $G'$ be the Gauss diagram 
obtained from $G$ by adding 
\begin{itemize}
\item
shells spanning $C_1$ 
such that the sum of signs are equal to $x$ and 
\item
shells spanning $C_2$ 
such that the sum of signs are equal to $-x$
\end{itemize}
to $\gamma$. 
Then the virtual link $L'$ presented by $G'$ 
has the same invariants as $L$ except 
$$\left\{
\begin{array}{l}
J_1(K_1';L')+J_{-\lambda+1}(K_1';L')
=J_1(K_1;L)+J_{-\lambda+1}(K_1;L)+x=a_1+a_{-\lambda+1} 
\mbox{ and}\\
J_1(K_2';L')+J_{\lambda+1}(K_2';L')
=J_1(K_2;L)+J_{\lambda+1}(K_2;L)-x
=b_1+b_{\lambda+1}. 
\end{array}\right.$$

Consider four kinds of portions of self-chords 
such that two of them span $C_1$ 
and the other two span $C_2$ as shown in Figure~\ref{fig602}. 
Adding the left two portions to $C_1$ suitably, 
we can change $J_1(K_1';L')$ and $J_{-\lambda+1}(K_1';L')$ 
arbitrarily with keeping the sum $a_1+a_{-\lambda+1}$ 
so that we realize $a_1$ and $a_{-\lambda+1}$, 
respectively. 
Similarly, 
adding the right two portions to $C_2$ suitably, 
we can change $J_1(K_2';L')$ and $J_{\lambda+1}(K_2';L')$ 
arbitrarily with keeping the sum $b_1+b_{\lambda+1}$ 
so that we realize $b_1$ and $b_{\lambda+1}$, 
respectively. 
\end{proof}

\begin{figure}[htb]
\begin{center}
\includegraphics[bb=0 0 312 37]{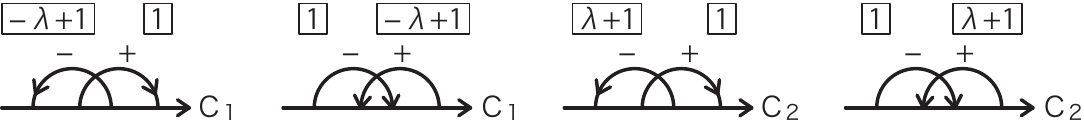}
\caption{Changing $J_1(K_1;L)$, $J_{-\lambda+1}(K_1+L)$, 
$J_1(K_2;L)$, and $J_{\lambda+1}(K_2;L)$}
\label{fig602}
\end{center}
\end{figure}


\end{document}